\documentclass{elsarticle}

\makeatletter
\def\ps@pprintTitle{%
 \let\@oddhead\@empty
 \let\@evenhead\@empty
 \def\@oddfoot{}%
 \let\@evenfoot\@oddfoot}
\makeatother

\usepackage{color} 
\usepackage{amssymb}
\usepackage{amsthm}
\usepackage{MnSymbol}
\usepackage{pgf,tikz}
\usepackage{caption}
\usepackage[french]{babel}

\usepackage{chngcntr}
\usepackage{lipsum}
\usetikzlibrary{arrows}

\usepackage{enumitem}

\newtheorem{theorem}{Theorem}[section]

\newtheorem{theoreme}{Th\'eor\`eme}[section]
\newtheorem{e-proposition}[theoreme]{Proposition}

\newtheorem{e-definition}[theoreme]{Definition\rm}

\newtheorem{lemme}[theoreme]{Lemme}

\newtheorem{corollaire}[theoreme]{Corollaire}

 \newtheorem{Fait}{\textbf{Fait}}[section]

\renewcommand{\theequation}{\arabic{equation}}
\usepackage{calrsfs} %calligraphie%%%%%%%%%%
\usepackage{ulem} % pour hachurer ,barrer souligner.. une phrase
\usepackage{cancel} %  barrer obliquement une lettre
\usepackage{pgf,tikz}
\usepackage{mathrsfs}
\usetikzlibrary{arrows}
\usepackage{latexsym}
\usepackage{graphicx}
\usepackage{caption} 
\usetikzlibrary{shapes}
\begin{document}
\nocite{*}
\sloppy

\begin{frontmatter}

\title{Morphologie des posets (-1)-critiques}
\author{Rachid Sahbani\fnref{}}
\address{Sfax University, Faculty of Sciences of Sfax, Sfax, Tunisia}
\ead{rachidsahbani5@gmail.com}

\begin{abstract}

Let $G=(V,A)$ be a digraph. For $X\subseteq V$, the subdigraph of $G$ induced by $X$ is denoted by $G[X]$. A subset $I$ of $V$ is an  interval  of $G$ if for every $a,b \in I$ and $x \in V \setminus I$, $(x,a) \in A$ if and only if $(x,b) \in A$, and similarly for $(a,x)$ and $(b,x)$. The trivial intervals of $G$ are $\varnothing$, $V$ and $\lbrace x\rbrace$, where $x\in V$. The digraph $G$ is indecomposable if $\vert V(G)\vert\geqslant 3$ and all its intervals are trivial. Given an indecomposable digraph $G$, a vertex $x$ of $G$ is critical, if the induced subdigraph $G[V(G) \setminus \{x\}]$ is decomposable. The digraph $G$ is said to be (-1)-critical if it admits a single non-critical vertex. A poset (or a strict partial order) is a transitive digraph. 
In this paper, We characterize the (-1)-critical posets.

\end{abstract}
\end{frontmatter}
\section*{Résumé}  

\'Etant donné un digraphe $G$, le sous-digraphe de $G$ induit par une partie $X$ de $V(G)$ est noté $G[X]$. Une partie $I$ de $V(G)$ est un intervalle de $G$ lorsque pour tous $a,b\in I$ et $x\in V(G)\setminus I$,  $(x,a) \in A(G)$ si et seulement si $(x,b) \in A(G)$, et de même pour  $(a,x)$ et $(b,x)$.  Par exemple, $\varnothing$, $V(G)$ et $\{x\}$ ($x\in V(G)$) sont des intervalles de $G$, appelés intervalles triviaux.  Un digraphe, à au moins 3 sommets,  est indécomposable lorsque tous ses intervalles sont triviaux, sinon il est décomposable. 
\'Etant donné un digraphe indécomposable  $G$, un sommet $x$ de $G$ est critique si le sous-digraphe $G[V(G)\setminus \{x\}]$ est décomposable. le  digraphe $G$ est dit (-1)-critique lorsqu'il  admet un seul sommet non critique. Un poset (ou ordre partiel (strict)) est un digraphe transitif.
Dans ce papier, nous caractérisons morphologiquement les posets (-1)-critiques.

 \section{Introduction}
 
 \subsection{\textbf{Terminologie}}
  \textit{Un graphe} $G$ est un couple $(V(G),E(G))$ de deux ensembles finis et disjoints tels que  $V(G)\neq \varnothing$ et $E(G)\subseteq \lbrace \lbrace x,y\rbrace:x\neq y\in V(G)\rbrace$.  $V(G)$ est l'ensemble des \textit{sommets} de $G$  et $E(G)$ est l'ensemble des \textit{arêtes} de $G$. On note $G=(V,E)$. Le graphe  \textit{ complémentaire} d'un graphe $G$  est le graphe  $\overline{G}$ défini sur le même ensemble de sommets $V(G)$ par $E(\overline{G})=\binom{V(G)}{2}\setminus E(G)$. Par exemple, étant donné un ensemble fini  $V$, le complémentaire du digraphe  complet $K_{V}=(V, (V \times V) \setminus \{(x,x): x \in V\})$ est le digraphe vide  $\overline{K_V} = (V,\varnothing)$. 
 
Un {\it digraphe }  $D$  est la donnée d'un ensemble fini $V = V(D)$ de {\it sommets} avec un ensemble $A = A(D)$  d'{\it arcs}, où un arc est un couple de sommets distincts. \'Etant donné un digraphe $D=(V,A)$, pour toute partie $X$ de $V$ est associé le sous-digraphe $D[X]=(X,A\cap(X\times X))$ {\it induit } par $X$. Le sous-digraphe induit $D[V \setminus X]$ est aussi noté $D - X$, et est noté $D - x$  lorsque $X = \{x\}$, où $x \in V$. Le \textit{dual}  $D^{\star}=(V,A^{\star})$ de $D$ est le digraphe défini sur $V$ comme suit: pour tous  $x\neq y\in V$, $(x,y)\in A^{\star}$\ si et seulement si\  $(y,x)\in A$. 
%Le digraphe $D$ est \textit{autodual} lorsque $G\simeq G^{\star}$.

\'Etant donnés deux digraphes (resp. graphes) $D=(V,A)$ (resp. $G=(V,E)$ ) et $D'=(V',A')$ (resp. $G'=(V',E')$), un \textit{isomorphisme} de $D$ (resp. $G$) sur $D'$ (resp. $G'$) est une bijection $f$ de $V$ sur $V'$ telle que pour tous $x\neq y\in V$, $(x,y)\in A$ (resp. $\lbrace x,y\rbrace\in E$ ) si et seulement si $(f(x),f(y))\in A'$ (resp. $\lbrace f(x),f(y)\rbrace\in E'$). Lorsqu'un tel isomorphisme existe, on dit que $D$ (resp. $G$) et $D'$ (resp. $G'$) sont \textit{isomorphes} et on note $D\simeq D'$ (resp. $G\simeq G'$).

Soient $D$ et $D'$ deux digraphes.  Le digraphe  $D$ \textit{abrite} le digraphe  $D'$  lorsqu'il existe une partie $X$ de $V(D)$ telle que $D'\simeq D[X]$. Lorsque  $D'$ ne s'abrite pas dans $D$, on dit que $D$ \textit{omet} $D'$. 

Un \textit{poset} (ou un ordre partiel strict) $\mathcal{O}$ est un digrpahe transitif, c'est à dire pour tous $x,y,z\in V(\mathcal{O})$, si $(x,y)\in A(\mathcal{O})$ et $(y,z)\in A(\mathcal{O})$, alors $(x,z)\in A(\mathcal{O})$. \`A chaque poset $\mathcal{O}$ est associé son \textit{graphe de comparabilité} 
$\overset{\longleftrightarrow}{\mathcal{ O}}$ défini sur  $V(\overset{\longleftrightarrow}{\mathcal{ O}})=V(\mathcal{O})$ comme suit. Pour tous sommets distincts $x$ et $y$ de $\overset{\longleftrightarrow}{\mathcal{ O}}$, $\{x,y\}\in E(\overset{\longleftrightarrow}{\mathcal{ O}})$ si $(x,y)\in A(\mathcal{O})$ ou $(y,x)\in A(\mathcal{O})$. Comme un arc d'un digraphe  est un  couple de sommets \underline{ distincts}, un  poset peut être encore défini comme étant une orientation transitive d'un graphe, c'est à dire, une orientation d'un graphe  
\footnote{$-$\'Etant donné un graphe $G$, une \textit{orientation} de $G$  est le digraphe obtenu a partir de $G$ en remplaçant chaque arête $\{x,y\}$ de $G$ par exactement un des deux arcs possibles $(x,y)$ et $(y,x)$.

$-$ Un \textit{tournoi}  est un digraphe $T$ tel que pour tous sommets distincts $x$ et $y$ de $V(T)$, $(x,y)\in A$ si et seulement si $(y,x)\notin A$. Un tournoi est alors une orientation d'un graphe complet.} de façon que le digraphe résultant soit transitif. Remarquons que  tous les graphes n'admettent pas d'orientations transitives (prendre par exemple le cycle $C_5$). Un \textit{graphe de comparabilité} est un graphe admettant des orientations transitives. Par exemple, le graphe complet $K_n$  défini sur $\{0,\ldots n-1\}$, où $n\geqslant 2$,  est un graphe de comparabilité. les $n!$ orientations transitives de $K_n$ sont des ordres totaux \footnote{  Un \textit{ordre total}  est un tournoi transitif. Pour tout entier $n\geqslant 2$, l'\textit{ordre total usuel}  $O_n$ est le tournoi défini sur $\{0,\ldots, n-1\}$ par $A(O_n)=\{(i,j): 0\leqslant i< j\leqslant n-1\}$.} définis sur $\{0,\ldots n-1\}$. 

Clairement, la transitivité dans la classe des digraphes est une propriété héréditaire. Il s'ensuit qu'un sous-digraphe (induit) d'un poset est un poset.

\subsection{\textbf{Préliminaires}}
Soit $D$ un digraphe. Une partie $X$ de $V(D)$ est un {\it intervalle} \cite{I,S.T} (ou {\it clan} \cite{ER} ou {\it module} \cite{Muller}) de $D$ si pour tous $a,b\in I$ et $x \in V(D) \setminus I$, on a  $(a,x)\in A(D)$ si et seulement si $(b,x)\in A(D)$, et de m\^eme pour $(x,a)$ et $(x,b)$. Clairement, $\varnothing$, $V(D)$ et ses singletons sont des intervalles de $D$, appelés intervalles {\it triviaux}.  Le digraphe $D$ est {\it indécomposable} \cite{I,S.T} (ou {\it premier} \cite{CH2} ou {\it primitif} \cite{ER}) lorsque tous ses intervalles sont triviaux; sinon il est {\it décomposable}. Clairement un digraphe $D$ (resp. un graphe $G$) et son dual $D^{\star}$ (resp. son complémentaire $\overline{G}$) partagent les mêmes intervalles. En particulier le digraphe $D$ (resp. graphe $G$) est indécomposable si et seulement si son dual $D^{\star}$ (resp. son complémentaire $\overline{G}$) est indécomposable.

Le théorème suivant est un résultat important sur  l'aspect
 héréditaire\\
  ascendant de l'indécomposabilité. 
\begin{theoreme} \cite{S.T} \label{ER +2}
\'Etant donné un digraphe indécomposable $D$, pour toute 
partie $X$ de $V(D)$ telle que $3 \leq |X| \leq |D| -2$ et $D[X]$ est indécomposable, il existe
$x \neq y \in V(D) \setminus X$ tels que $D[X \cup \{x,y\}]$ est indécomposable.
\end{theoreme}
Comme tout digraphe indécomposable, à au moins trois sommets, abrite un digraphe indécomposable à trois ou quatre sommets \cite{ER}, nous obtenons, en appliquant Théorème~\ref{ER +2} plusieurs fois,  ce que suit.  

\begin{corollaire} \label{-1-2}
\textit{Tout digraphe indécomposable d'ordre $n$ ($n \geqslant 5$) abrite un digraphe indécomposable d'ordre $n-1$ ou $n-2$}.
\end{corollaire}
Ce dernier résultat amène alors à introduire la notion de criticité. Un sommet {\it critique} d'un digraphe indécomposable $D$ est  un sommet $x$ de $D$ tel que le digraphe $D-x$ est décomposable.  \'Etant donné un digraphe indécomposable $D$, à au moins 4 sommets, $D$ est {\it critique} \cite{Boniz,S.T}  lorsque tous ses sommets sont critiques. Les digraphes critiques ont été caractérisés indépendamment  par Schmerl et Trotter en 1993 \cite{S.T}, et par  Bonizonni en 1994 \cite{Boniz}.

 En 2009, Boudabbous et Ille \cite{BI} ont donné un autre type d'hérédité de la manière suivante.
 \begin{theoreme} \cite{BI}   \'Etant donné un digraphe indécomposable $D$ à au moins sept sommets et admettant au moins deux sommets non critiques, $D$ admet un sommet $x$ tel que $D-x$ est aussi indécomposable est non critique.
 \end{theoreme} 
  Ceci indique que l'indécomposabilité non critique est une propriété héréditaire pour les digraphes indécomposables (à au moins 7 sommets) admettant au moins deux sommets non critiques. 
   Ce dernier  résultat avait amené ses auteurs à poser le problème de caractérisation des digraphes indécomposables admettant un unique sommet non critique,  appelés digraphes {\it (-1)-critiques} \cite{BBD article,houmem}. En 2015, les auteurs de \cite{houmem} ont donné un procédé de construction des digraphes (-1)-critiques sans plus de précision sur leurs morphologie. En 2007, les auteurs de \cite{BBD article} ont donné une caractérisation morphologique des tournois \\
   (-1)-critiques. En 2017, avec une relecture de \cite{houmem}, H. Belkhechine a donné  une caractérisation morphologique des graphes (-1)-critiques \cite{graphes (-1)-critiques}. Tout comme les tournois et  les graphes, les posets  (-1)-critiques présentent un intér\^et particulier. Nous donnons, dans cet article, une caractérisation morphologique complète des posets (-1)-critiques.

 Le théorème suivant est fondamental pour notre caractérisation des posets 
(-1)-critiques. 
\begin{theoreme}(Gallai \cite{strong interval, strong intervals})\label{poset indec ssi comp indec}
Un poset $\mathcal{O}$ est indécomposable si et seulement si son graphe de comparabilité $\overset{\longleftrightarrow}{\mathcal{ O}}$ est indécomposable.
\end{theoreme}

Comme les graphes à 3 sommets sont tous décomposables, d'après le théorème~\ref{poset indec ssi comp indec}, tous les posets à 3 sommets sont décomposables. De plus, comme, à isomorphisme près, le seul graphe indécomposable à 4 sommets est le chemin $P_4= (\{0,1,2,3\}, \{\{0,1\}, \{1,2\}, \{2,3\}\})$, d'après  Théorème~\ref{poset indec ssi comp indec}, à isomorphisme près, les seuls posets indécomposables à 4 sommets sont $\overrightarrow{P_4}$ et son dual $\overrightarrow{P_4}^{\star}$, où $\overrightarrow{P_4}$ est le poset défini  sur  $\{0,1,2,3,\}$ par  $A(\overrightarrow{P_4}) = \{(0,1), (2,1), (2,3)\}$. Clairement, ces deux derniers posets sont critiques. Nous disons alors qu'un poset indécomposable est (-1)-critique, lorsqu'il est à au moins 5 sommets, et ne possède qu'un seul sommet non critique. \'Etant donné un  poset indécomposable $\mathcal{O}$ et un sommet $x$ de $\mathcal{O}$, le poset $\mathcal{O}$ est \textit{(-1)-critique en $x$}\index{poset (-1)-critique en  $x$} si $x$ est le seul sommet non critique de $\mathcal{O}$.
 
De l'hérédité de la transitivité, le théorème~\ref{poset indec ssi comp indec} implique le résultat suivant.
\begin{corollaire}\label{corollaire des posets (-1)-critiques} Soit  $\mathcal{O}$ un poset indécomposable à au moins 5 sommets et soit $x$ un sommet de $\mathcal{O}$. Le poset $\mathcal{O}$ est (-1)-critique en $x$ si et seulement si son graphe de comparabilité $\overset{\longleftrightarrow}{\mathcal{ O}}$ est (-1)-critique en $x$.
\end{corollaire}

  Il en résulte du corollaire~\ref{corollaire des posets (-1)-critiques}, que les posets $(-1)$-critiques sont les orientations transitives des graphes $(-1)$-critiques qui sont de comparabilité. Il faut donc parmi les graphes $(-1)$-critiques voir ceux qui sont de comparabilité pour une éventuelle caractérisation des posets $(-1)$-critiques. 
  Nous rappelons alors, dans ce que suit, la caractérisation des graphes (-1)-critiques~\cite{graphes (-1)-critiques}. 

Pour tout entier $n \geqslant 2$, le graphe $G_{2n}$ désigne le graphe critique~\cite{S.T}  défini sur $\{0, \ldots, 2n-1\}$ par $E(G_{2n}) = \{\{2i,2j+1\} : 0 \leq i \leq j \leq n-1 \}$.
  Le graphe {\it scindé}\footnote{$-$Un graphe  est dit scindé (en anglais: split graph) lorsque l'ensemble de ses sommets 
  
 \hspace{0.5cm} peut être partitionné en une clique et un 
 stable.
  
  \vspace{0.25cm}
 \hspace{0.18cm} $-$ Les sommets impairs forment une clique de $G'_{2n}$ et les sommets pairs en forment un 
 
 \hspace{0.5cm} stable.  }  $G'_{2n}$  associé au graphe critique $G_{2n}$  intervient dans la description des graphes 
  (-1)-critiques. Le graphe scindé $G'_{2n}$  est le graphe défini sur $\{0, \ldots, 2n-1\}$ par: $E(G'_{2n}) = E(G_{2n}) \cup \{\{2p-1, 2q-1\}: 1\leq p < q \leq n\}.$ 
  
   Pour tout entier naturel $n$, désignons par $t_{n}$  l'application injective de $\mathbb{N}$ sur $\mathbb{N}$, définie par: $p\longmapsto p+n$. \'Etant donn\'e un digraphe $D$ dont les sommets sont des entiers naturels, on note $t_n(D)$ le digraphe d\'efini sur  $t_n(V(D))$ par : $ A(t_n(D)) = \{(t_n(x), t_n(y)): (x,y) \in A(D)\}$.
  %\vspace{0.5cm}
  
 Maintenant, Considérons $k$  entiers strictement positifs  $n_1, \ldots, n_k$ ($k\geqslant 2$), posons $s_0 = 0$, et pour tout $i \in \{1, \ldots, k\}$, $s_i = n_1 + \cdots +n_i$.  Les trois graphes $G_{2n_1, \ldots, 2n_k}$, $G'_{2n_1, \ldots, 2n_k}$ et $H_{1, 2n_1, \ldots, 2n_k}$ sont alors définis comme suit \cite{graphes (-1)-critiques}.
 
$\bullet$ Le graphe $G_{2n_1, \ldots, 2n_k}$ est d\'efini sur $\{0, \ldots, 2s_k-1\}$ par  
 $$E(G_{2n_1, \ldots, 2n_k}) = \bigcup\limits_{i=0}^{k-1}E(t_{2s_{i}}(G_{2n_{i+1}})) \cup \{\{2p-1,2q-1\}: 1 \leq p \leq n_1 \ \textnormal{et} \ n_1 +1 \leq q \leq s_k\}.$$ 

$\bullet$ Le graphe $G'_{2n_1, \ldots, 2n_k}$ est obtenu à partir du graphe $G_{2n_1, \ldots, 2n_k}$ en remplaçant son sous-graphe (induit) $G_{2n_1}$ par $G'_{2n_1}$. Plus précisément, $G'_{2n_1, \ldots, 2n_k}$ est le graphe défini sur $\{0, \ldots, 2s_k-1\}$ par $$E(G'_{2n_1, \ldots, 2n_k}) = E(G_{2n_1, \ldots, 2n_k}) \cup \{\{2p-1, 2q-1\}: 1\leq p < q \leq n_1\}.$$ 

$\bullet$ Le graphe scindé \footnote{Les sommets pairs forment une clique de $H_{1, 2n_1, \ldots, 2n_k}$ et les sommets impairs en forment un stable.} $H_{1, 2n_1, \ldots, 2n_k}$ défini sur $\{0, \ldots, 2s_k\}$ par 
$$E(H_{1, 2n_1, \ldots, 2n_k}) = \bigcup\limits_{i=0}^{k-1}E(t_{2s_{i}+1}(G'_{2n_{i+1}})) \cup \{\{2p,2q\}: 0 \leq p < q \leq s_k\} .$$

Enfin, introduisons les familles suivantes: 
 $$\mathcal{G} = \{G_{2n_1, \ldots, 2n_k}: \ k \geqslant 3 \ \textnormal{et pour tout} \ i \in \{1, \ldots, k\},  \ n_i \geqslant 1\}, $$
$$\mathcal{G'} = \{G'_{2n_1, \ldots, 2n_k}: \ k \geqslant 2, \ n_1 \geqslant 2, \ \textnormal{et pour tout} \ i \in \{2, \ldots, k\}, \  n_i \geqslant 1\}, {\rm \ et}$$ 
$$\mathcal{H} = \{H_{1, 2n_1, \ldots, 2n_k}: k \geqslant 2 \ \textnormal{et pour tout} \ i \in \{1, \ldots, k\}, \  n_i \geqslant 1\}.$$ 

 \begin{theoreme} \label{graphes (-1)-critiques}\cite{graphes (-1)-critiques}
 \`A isomorphisme près, les graphes (-1)-critiques sont les graphes de $\mathcal{G} \cup \mathcal{G'} \cup \mathcal{H}$ et leurs complémentaires, le sommet $0$ étant l'unique sommet non critique de chacun de ces graphes. 
\end{theoreme}
\section{Caractérisation des posets (-1)-critiques}

Dans cette section, nous donnons une caractérisation morphologique des posets (-1)-critiques. Nous introduisons alors pour tout entier $n \geqslant 2$, le poset critique  $Q_{2n}$ \cite{S.T}  d\'efini sur $\{0, \ldots, 2n-1\}$ par $A(Q_{2n}) = \{(2p,2q+1) : 0 \leqslant p \leqslant q \leq n-1 \}$. Le poset  $Q'_{2n}$  est défini sur $\{0, \ldots, 2n-1\}$ par $A(Q'_{2n})= A(Q_{2n})\cup \{ (2p-1,2q-1): 1\leqslant p< q\leqslant n \}$ (voir Figure \ref{posets Q et Q'}).

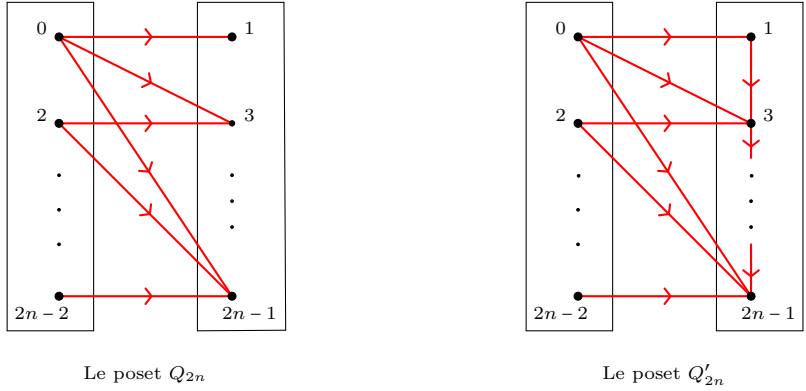
\begin{figure}[h]
\begin{center}
\definecolor{ffqqqq}{rgb}{1,0,0}
\begin{tikzpicture}[line cap=round,line join=round,>=triangle 45,x=1.15cm,y=1.15cm]
%\clip(-5.326666666666668,-3.5533333333333315) rectangle (5.82,6.11333333333333);
\draw [line width=0.4pt] (-4.6,3.4)-- (-3.6,3.4);
\draw [line width=0.4pt] (-3.6,3.4)-- (-3.6,-0.4);
\draw [line width=0.4pt] (-3.6,-0.4)-- (-4.6,-0.4);
\draw [line width=0.4pt] (-4.6,-0.4)-- (-4.6,3.4);
\draw [line width=0.4pt] (-2.4,3.4)-- (-1.4,3.4);
\draw [line width=0.4pt] (-1.4,3.4)-- (-1.38,-0.39333333333333304);
\draw [line width=0.4pt] (-1.38,-0.39333333333333304)-- (-2.4,-0.4);
\draw [line width=0.4pt] (-2.4,-0.4)-- (-2.4,3.4);
\draw [line width=0.4pt] (1.4,3.4)-- (2.4,3.4);
\draw [line width=0.4pt] (2.4,3.4)-- (2.4,-0.4);
\draw [line width=0.4pt] (2.4,-0.4)-- (1.4,-0.4);
\draw [line width=0.4pt] (1.4,-0.4)-- (1.4,3.4);
\draw [line width=0.4pt] (3.6,3.4)-- (4.6,3.4);
\draw [line width=0.4pt] (4.6,3.4)-- (4.6,-0.4);
\draw [line width=0.4pt] (4.6,-0.4)-- (3.6,-0.4);
\draw [line width=0.4pt] (3.6,-0.4)-- (3.6,3.4);
\draw [line width=0.8pt,color=ffqqqq] (2,3)-- (4,3);
\draw [line width=0.8pt,color=ffqqqq] (3.07,3) -- (3,2.91);
\draw [line width=0.8pt,color=ffqqqq] (3.07,3) -- (3,3.09);
\draw [line width=0.8pt,color=ffqqqq] (2,3)-- (4,2);
\draw [line width=0.8pt,color=ffqqqq] (3.0626099033699945,2.468695048315003) -- (2.959750776405004,2.4195015528100074);
\draw [line width=0.8pt,color=ffqqqq] (3.0626099033699945,2.468695048315003) -- (3.0402492235949974,2.5804984471899926);
\draw [line width=0.8pt,color=ffqqqq] (2,2)-- (4,2);
\draw [line width=0.8pt,color=ffqqqq] (3.07,2) -- (3,1.91);
\draw [line width=0.8pt,color=ffqqqq] (3.07,2) -- (3,2.09);
\draw [line width=0.8pt,color=ffqqqq] (2,2)-- (4,0);
\draw [line width=0.8pt,color=ffqqqq] (3.0494974746830583,0.9505025253169419) -- (2.936360389693211,0.9363603896932108);
\draw [line width=0.8pt,color=ffqqqq] (3.0494974746830583,0.9505025253169419) -- (3.0636396103067884,1.0636396103067896);
\draw [line width=0.8pt,color=ffqqqq] (2,0)-- (4,0);
\draw [line width=0.8pt,color=ffqqqq] (3.07,0) -- (3,-0.09);
\draw [line width=0.8pt,color=ffqqqq] (3.07,0) -- (3,0.09);
\draw [line width=0.8pt,color=ffqqqq] (2,3)-- (4,0);
\draw [line width=0.8pt,color=ffqqqq] (3.038829013735766,1.4417564793963507) -- (2.925115473509594,1.4500769823397293);
\draw [line width=0.8pt,color=ffqqqq] (3.038829013735766,1.4417564793963507) -- (3.0748845264904054,1.5499230176602705);
\draw [line width=0.8pt,color=ffqqqq] (-4,3)-- (-2,3);
\draw [line width=0.8pt,color=ffqqqq] (-2.93,3) -- (-3,2.91);
\draw [line width=0.8pt,color=ffqqqq] (-2.93,3) -- (-3,3.09);
\draw [line width=0.8pt,color=ffqqqq] (-4,3)-- (-2,2);
\draw [line width=0.8pt,color=ffqqqq] (-2.9373900966300055,2.468695048315003) -- (-3.0402492235949956,2.4195015528100074);
\draw [line width=0.8pt,color=ffqqqq] (-2.9373900966300055,2.468695048315003) -- (-2.9597507764050035,2.5804984471899926);
\draw [line width=0.8pt,color=ffqqqq] (-4,3)-- (-2,0);
\draw [line width=0.8pt,color=ffqqqq] (-2.961170986264234,1.4417564793963507) -- (-3.074884526490406,1.4500769823397293);
\draw [line width=0.8pt,color=ffqqqq] (-2.961170986264234,1.4417564793963507) -- (-2.9251154735095937,1.5499230176602705);
\draw [line width=0.8pt,color=ffqqqq] (-4,2)-- (-2,2);
\draw [line width=0.8pt,color=ffqqqq] (-2.93,2) -- (-3,1.91);
\draw [line width=0.8pt,color=ffqqqq] (-2.93,2) -- (-3,2.09);
\draw [line width=0.8pt,color=ffqqqq] (-4,2)-- (-2,0);
\draw [line width=0.8pt,color=ffqqqq] (-2.9505025253169417,0.9505025253169419) -- (-3.0636396103067893,0.9363603896932108);
\draw [line width=0.8pt,color=ffqqqq] (-2.9505025253169417,0.9505025253169419) -- (-2.9363603896932102,1.0636396103067896);
\draw [line width=0.8pt,color=ffqqqq] (-4,0)-- (-2,0);
\draw [line width=0.8pt,color=ffqqqq] (-2.93,0) -- (-3,-0.09);
\draw [line width=0.8pt,color=ffqqqq] (-2.93,0) -- (-3,0.09);
\draw [line width=0.8pt,color=ffqqqq] (4,3)-- (4,2);
\draw [line width=0.8pt,color=ffqqqq] (4,2.43) -- (3.91,2.5);
\draw [line width=0.8pt,color=ffqqqq] (4,2.43) -- (4.09,2.5);
\draw [line width=0.8pt,color=ffqqqq] (4,0.6)-- (4,0);
\draw [line width=0.8pt,color=ffqqqq] (4,0.23) -- (3.91,0.3);
\draw [line width=0.8pt,color=ffqqqq] (4,0.23) -- (4.09,0.3);
\draw [line width=0.8pt,color=ffqqqq] (4,2)-- (4,1.6);
\draw [line width=0.8pt,color=ffqqqq] (4,1.73) -- (3.91,1.8);
\draw [line width=0.8pt,color=ffqqqq] (4,1.73) -- (4.09,1.8);
\begin{scriptsize}

\draw [fill=black] (-4,1.4) circle (0.5pt);
\draw [fill=black] (-4,1) circle (0.5pt);
\draw [fill=black] (-4,0.6) circle (0.5pt);
\draw [fill=black] (-2,1.1) circle (0.5pt);
\draw [fill=black] (-2,0.8) circle (0.5pt);
\draw [fill=black] (-2,1.4) circle (0.5pt);
\draw [fill=black] (2,3) circle (1.5pt);
%\draw[color=black] (2.1533333333333338,3.28) node {$I_1$};
\draw [fill=black] (4,3) circle (1.5pt);
%\draw[color=black] (4.153333333333334,3.28) node {$J_1$};
\draw [fill=black] (2,2) circle (1.5pt);
%\draw[color=black] (2.1533333333333338,2.28) node {$K_1$};
\draw [fill=black] (4,2) circle (1.5pt);
%\draw[color=black] (4.153333333333334,2.28) node {$L_1$};
\draw [fill=black] (2,0) circle (1.5pt);
%\draw[color=black] (2.1533333333333338,0.2666666666666665) node {$M_1$};
\draw [fill=black] (4,0) circle (1.5pt);
%\draw[color=black] (4.153333333333334,0.2666666666666665) node {$N_1$};
\draw [fill=black] (2.0066666666666673,1.3933333333333324) circle (0.5pt);
\draw [fill=black] (4,0.8) circle (0.5pt);
\draw [fill=black] (4,2) circle (0.5pt);
\draw [fill=black] (4,1.1) circle (0.5pt);
\draw [fill=black] (2,1) circle (0.5pt);
\draw [fill=black] (2,0.6) circle (0.5pt);
\draw [fill=black] (4.006666666666668,1.3933333333333324) circle (0.5pt);

\draw [fill=black] (-4.,3.) circle (1.5pt);
\draw[color=black] (-4.2,3.1) node {$0$};
\draw [fill=black] (-2.,3.) circle (1.5pt);
\draw[color=black] (-1.8,3.1) node {$1$};
\draw [fill=black] (-4.,2.) circle (1.5pt);
\draw[color=black] (-4.2,2.1) node {$2$};
\draw [fill=black] (-2.,2.) circle (1.0pt);
\draw[color=black] (-1.8,2.1) node {$3$};
\draw [fill=black] (-4.,0.) circle (1.5pt);
\draw[color=black] (-4.2,-0.2) node {$2n-2$};
\draw [fill=black] (-2.,0.) circle (1.5pt);
\draw[color=black] (-1.8,-0.2) node {$2n-1$};

\draw[color=black] (1.8,3.1) node {$0$};
%\draw [fill=black] (4.,3.) circle (1.5pt);
\draw[color=black] (4.2,3.1) node {$1$};
%\draw [fill=black] (2.,1.) circle (1.5pt);
\draw[color=black] (1.8,2.1) node {$2$};
%\draw [fill=black] (4.,1.) circle (1.5pt);
\draw[color=black] (4.2,2.1) node {$3$};
%\draw [fill=black] (2.,0.) circle (1.5pt);
\draw[color=black] (1.8,-0.2) node {$2n-2$};
%\draw [fill=black] (4.,0.) circle (1.5pt);
\draw[color=black] (4.2,-0.2) node {$2n-1$};
\draw[color=black] (3,-0.9) node {Le poset $ Q'_{2n}$};
\draw[color=black] (-3,-0.9) node {Le poset $Q_{2n}$};
\end{scriptsize}

\end{tikzpicture}
\caption{\textbf{Les posets   $Q_{2n}$ et $Q'_{2n}$}.}  
\label{posets Q et Q'}
\end{center} 
\end{figure}

Maintenant, considérons   $k$  entiers strictement positifs  $n_1, \ldots, n_k$ ($k\geqslant 2$), et posons pour tout $i \in \{1, \ldots, k\}$, $s_i = n_1 + \cdots +n_i$.  Les trois digraphes $Q_{2n_1, \ldots, 2n_k}$, $Q'_{2n_1, \ldots, 2n_k}$ et $R_{1, 2n_1, 2n_2}$, définis ci-après (voir aussi Figures \ref{poset (-1)-crit Q}, \ref{poset (-1)-crit Q'} et \ref{poset (-1)-crit H}), interviennent dans la caractérisation des posets (-1)-critiques.

 $\bullet$ Le digraphe $Q_{2n_1, \ldots, 2n_k}$ est d\'efini sur $\{0, \ldots, 2s_k-1\}$ par $A(Q_{2n_1, \ldots, 2n_k}) = $ 
 $$\bigcup\limits_{i=1}^{k-1}A(t_{2s_{i}}((Q_{2n_{i+1}})^{\star}))\cup A(Q_{2n_1}) \cup \{(2q-1,2p-1): 1 \leqslant p \leqslant n_1,   n_1 +1 \leqslant q \leqslant s_k\}.$$ 

$\bullet$ Le digraphe $Q'_{2n_1, \ldots, 2n_k}$ est  défini sur $\{0, \ldots, 2s_k-1\}$ par $$A(Q'_{2n_1, \ldots, 2n_k}) = A(Q_{2n_1, \ldots, 2n_k}) \cup \{(2p-1, 2q-1): 1\leqslant p < q \leqslant n_1\}.$$ 

$\bullet$ Le digraphe $R_{1, 2n_1,  2n_2}$ défini sur $\{0, \ldots, 2s_2\}$ par $A(R_{1, 2n_1, 2n_2}) =$
$$  A(t_1(Q'_{2n_1}))\cup A(t_{2n_1+1}((Q'_{2n_2})^{\star})  )\cup  \{(0,2p): 0 \leqslant p \leqslant n_1\} \cup  \{(2q,2p): 0 \leqslant p \leqslant n_1, n_1+1\leqslant q \leqslant s_2\}.$$

Nous vérifions facilement, les résultats suivants.
\begin{Fait} \label{le post -1critique Q} Pour tous entiers strictement positifs $n_1, \ldots, n_k$ tels que  $k\geqslant 3$, le digraphe $Q_{2n_1, \ldots, 2n_k}$ est une orientation transitive du graphe (-1)-critique $G_{2n_1, \ldots, 2n_k}$.
\end{Fait}
\begin{Fait}\label{le post -1critique Q'} Pour tous entiers strictement positifs $n_1, \ldots, n_k$ tels que  $k\geqslant 2$ et $n_1\geqslant 2$, le digraphe $Q'_{2n_1, \ldots, 2n_k}$ est une orientation transitive du graphe (-1)-critique $G'_{2n_1, \ldots, 2n_k}$.
\end{Fait}
 
\begin{figure}[h]
\begin{center}
\definecolor{sqsqsq}{rgb}{0.12549019607843137,0.12549019607843137,0.12549019607843137}
\definecolor{ffqqqq}{rgb}{1.,0.,0.}
\definecolor{qqqqff}{rgb}{0.,0.,1.}
\begin{tikzpicture}[line cap=round,line join=round,>=triangle 45,x=1.5cm,y=1.5cm]
%\clip(-8.166666666666668,-6.526666666666663) rectangle (13.886666666666667,4.513333333333331);
\draw [line width=0.4pt] (-4.6,3.4)-- (-3.6,3.4);
\draw [line width=0.4pt] (-3.6,3.4)-- (-3.6,-0.4);
\draw [line width=0.4pt] (-3.6,-0.4)-- (-4.6,-0.4);
\draw [line width=0.4pt] (-4.6,-0.4)-- (-4.6,3.4);
\draw [line width=0.8pt,color=qqqqff] (-2.4,3.4)-- (-1.4,3.4);
\draw [line width=0.8pt,color=qqqqff] (-1.4,3.4)-- (-1.38,-0.39333333333333304);
\draw [line width=0.8pt,color=qqqqff] (-1.38,-0.39333333333333304)-- (-2.4,-0.4);
\draw [line width=0.8pt,color=qqqqff] (-2.4,-0.4)-- (-2.4,3.4);
\draw [line width=0.8pt,color=qqqqff] (0.4,3.4)-- (1.4,3.4);
\draw [line width=0.8pt,color=qqqqff] (1.4,3.4)-- (1.4,-0.4);
\draw [line width=0.8pt,color=qqqqff] (1.4,-0.4)-- (0.4,-0.4);
\draw [line width=0.8pt,color=qqqqff] (0.4,-0.4)-- (0.4,3.4);
\draw [line width=0.4pt] (2.6,3.4)-- (3.6,3.4);
\draw [line width=0.4pt] (3.6,3.4)-- (3.6,-0.4);
\draw [line width=0.4pt] (3.6,-0.4)-- (2.6,-0.4);
\draw [line width=0.4pt] (2.6,-0.4)-- (2.6,3.4);
\draw [line width=0.8pt,color=ffqqqq] (1.,3.)-- (3.,3.);
\draw [line width=0.8pt,color=ffqqqq] (2.07,3.) -- (2.,2.91);
\draw [line width=0.8pt,color=ffqqqq] (2.07,3.) -- (2.,3.09);
\draw [line width=0.8pt,color=ffqqqq] (1.,3.)-- (3.,1.);
\draw [line width=0.8pt,color=ffqqqq] (2.049497474683058,1.9505025253169417) -- (1.9363603896932107,1.9363603896932107);
\draw [line width=0.8pt,color=ffqqqq] (2.049497474683058,1.9505025253169417) -- (2.063639610306788,2.0636396103067893);
\draw [line width=0.8pt,color=ffqqqq] (1.,1.)-- (3.,1.);
\draw [line width=0.8pt,color=ffqqqq] (2.07,1.) -- (2.,0.91);
\draw [line width=0.8pt,color=ffqqqq] (2.07,1.) -- (2.,1.09);
\draw [line width=0.8pt,color=ffqqqq] (1.,1.)-- (3.,0.);
\draw [line width=0.8pt,color=ffqqqq] (2.062609903369994,0.4686950483150028) -- (1.9597507764050035,0.4195015528100071);
\draw [line width=0.8pt,color=ffqqqq] (2.062609903369994,0.4686950483150028) -- (2.040249223594997,0.5804984471899923);
\draw [line width=0.8pt,color=ffqqqq] (1.,0.)-- (3.,0.);
\draw [line width=0.8pt,color=ffqqqq] (2.07,0.) -- (2.,-0.09);
\draw [line width=0.8pt,color=ffqqqq] (2.07,0.) -- (2.,0.09);
\draw [line width=0.8pt,color=ffqqqq] (1.,3.)-- (3.,0.);
\draw [line width=0.8pt,color=ffqqqq] (2.0388290137357656,1.4417564793963507) -- (1.9251154735095939,1.4500769823397293);
\draw [line width=0.8pt,color=ffqqqq] (2.0388290137357656,1.4417564793963507) -- (2.074884526490405,1.5499230176602705);
\draw [line width=0.8pt,color=ffqqqq] (-4.,3.)-- (-2.,3.);
\draw [line width=0.8pt,color=ffqqqq] (-2.93,3.) -- (-3.,2.91);
\draw [line width=0.8pt,color=ffqqqq] (-2.93,3.) -- (-3.,3.09);
\draw [line width=0.8pt,color=ffqqqq] (-4.,3.)-- (-2.,2.);
\draw [line width=0.8pt,color=ffqqqq] (-2.9373900966300064,2.468695048315003) -- (-3.040249223594997,2.419501552810008);
\draw [line width=0.8pt,color=ffqqqq] (-2.9373900966300064,2.468695048315003) -- (-2.9597507764050044,2.5804984471899926);
\draw [line width=0.8pt,color=ffqqqq] (-4.,3.)-- (-2.,0.);
\draw [line width=0.8pt,color=ffqqqq] (-2.961170986264234,1.4417564793963507) -- (-3.074884526490406,1.4500769823397293);
\draw [line width=0.8pt,color=ffqqqq] (-2.961170986264234,1.4417564793963507) -- (-2.9251154735095937,1.5499230176602705);
\draw [line width=0.8pt,color=ffqqqq] (-4.,2.)-- (-2.,2.);
\draw [line width=0.8pt,color=ffqqqq] (-2.93,2.) -- (-3.,1.91);
\draw [line width=0.8pt,color=ffqqqq] (-2.93,2.) -- (-3.,2.09);
\draw [line width=0.8pt,color=ffqqqq] (-4.,2.)-- (-2.,0.);
\draw [line width=0.8pt,color=ffqqqq] (-2.9505025253169417,0.9505025253169419) -- (-3.0636396103067898,0.9363603896932108);
\draw [line width=0.8pt,color=ffqqqq] (-2.9505025253169417,0.9505025253169419) -- (-2.9363603896932107,1.0636396103067896);
\draw [line width=0.8pt,color=ffqqqq] (-4.,0.)-- (-2.,0.);
\draw [line width=0.8pt,color=ffqqqq] (-2.93,0.) -- (-3.,-0.09);
\draw [line width=0.8pt,color=ffqqqq] (-2.93,0.) -- (-3.,0.09);
\draw [line width=1.2pt,color=qqqqff] (0.4,1.0333333333333325)-- (-1.387451833091392,1.0200310096673564);
\draw [line width=1.2pt,color=qqqqff] (-0.5737237012674723,1.026086823192083) -- (-0.49447010193102264,1.1266794024025664);
\draw [line width=1.2pt,color=qqqqff] (-0.5737237012674723,1.026086823192083) -- (-0.49298173116036803,0.9266849405981227);
\draw [line width=0.8pt,color=qqqqff] (0.4,-2.)-- (1.4,-2.);
\draw [line width=0.8pt,color=qqqqff] (1.4,-2.)-- (1.4,-5.8);
\draw [line width=0.8pt,color=qqqqff] (1.4,-5.8)-- (0.4,-5.8);
\draw [line width=0.8pt,color=qqqqff] (0.4,-5.8)-- (0.4,-2.);
\draw [line width=1.2pt,color=qqqqff] (0.7933333333333333,-2.)-- (-1.3874518330913919,1.0200310096673564);
\draw [line width=1.2pt,color=qqqqff] (-0.34389357758033995,-0.42512660629560567) -- (-0.21598688879063352,-0.4314415855396833);
\draw [line width=1.2pt,color=qqqqff] (-0.34389357758033995,-0.42512660629560567) -- (-0.37813161096742537,-0.5485274047929606);
\draw [line width=0.4pt] (2.6,-2.)-- (3.6,-2.);
\draw [line width=0.4pt] (3.6,-2.)-- (3.6,-5.8);
\draw [line width=0.4pt] (3.6,-5.8)-- (2.6,-5.8);
\draw [line width=0.4pt] (2.6,-5.8)-- (2.6,-2.);
\draw [line width=0.8pt,color=ffqqqq] (1.,-2.4)-- (3.006666666666667,-2.4066666666666654);
\draw [line width=0.8pt,color=ffqqqq] (2.0733329470273296,-2.403565890189458) -- (2.0030343316611714,-2.493332836654184);
\draw [line width=0.8pt,color=ffqqqq] (2.0733329470273296,-2.403565890189458) -- (2.0036323350054954,-2.3133338300124797);
\draw [line width=0.8pt,color=ffqqqq] (1.,-2.4)-- (3.,-4.4);
\draw [line width=0.8pt,color=ffqqqq] (2.049497474683058,-3.4494974746830587) -- (1.9363603896932107,-3.463639610306789);
\draw [line width=0.8pt,color=ffqqqq] (2.049497474683058,-3.4494974746830587) -- (2.063639610306788,-3.3363603896932115);
\draw [line width=0.8pt,color=ffqqqq] (1.,-2.4)-- (3.,-5.4);
\draw [line width=0.8pt,color=ffqqqq] (2.0388290137357656,-3.95824352060365) -- (1.9251154735095939,-3.949923017660272);
\draw [line width=0.8pt,color=ffqqqq] (2.0388290137357656,-3.95824352060365) -- (2.074884526490405,-3.850076982339731);
\draw [line width=0.8pt,color=ffqqqq] (1.,-4.4)-- (3.,-4.4);
\draw [line width=0.8pt,color=ffqqqq] (2.07,-4.4) -- (2.,-4.49);
\draw [line width=0.8pt,color=ffqqqq] (2.07,-4.4) -- (2.,-4.31);
\draw [line width=0.8pt,color=ffqqqq] (1.,-4.4)-- (3.,-5.4);
\draw [line width=0.8pt,color=ffqqqq] (2.062609903369994,-4.931304951684997) -- (1.9597507764050035,-4.980498447189992);
\draw [line width=0.8pt,color=ffqqqq] (2.062609903369994,-4.931304951684997) -- (2.040249223594997,-4.819501552810008);
\draw [line width=0.8pt,color=ffqqqq] (1.,-5.4)-- (3.,-5.4);
\draw [line width=0.8pt,color=ffqqqq] (2.07,-5.4) -- (2.,-5.49);
\draw [line width=0.8pt,color=ffqqqq] (2.07,-5.4) -- (2.,-5.31);
\begin{scriptsize}
\draw [fill=black] (-4.,3.) circle (2.5pt);
\draw[color=black] (-4.2,3.1) node {$0$};
\draw [fill=black] (-2.,3.) circle (1.5pt);
\draw[color=black] (-1.8,3.1) node {$1$};
\draw [fill=black] (-4.,2.) circle (1.5pt);
\draw[color=black] (-4.2,2.1) node {$2$};
\draw [fill=black] (-2.,2.) circle (1.0pt);
\draw[color=black] (-1.8,2.1) node {$3$};
\draw [fill=black] (-4.,0.) circle (1.5pt);
\draw[color=black] (-4.2,-0.2) node {$2n_1-2$};
\draw [fill=black] (-2.,0.) circle (1.5pt);
\draw[color=black] (-1.8,-0.2) node {$2n_1-1$};

\draw [fill=black] (-0.4,0.1) circle (0.5pt);
\draw [fill=black] (-0.4,0.3) circle (0.5pt);
\draw [fill=black] (-0.4,0.5) circle (0.5pt);

\draw [fill=black] (-4.,1.4) circle (0.5pt);
\draw [fill=black] (-4.,1.) circle (0.5pt);
\draw [fill=black] (-4.,0.6) circle (0.5pt);
\draw [fill=black] (-2.,1.1) circle (0.5pt);
\draw [fill=black] (-2.,0.8) circle (0.5pt);
\draw [fill=black] (-2.,1.4) circle (0.5pt);
\draw [fill=black] (1.,3.) circle (1.5pt);
\draw[color=black] (0.9,3.15) node {$2s_2-1$};
\draw [fill=black] (3.,3.) circle (1.5pt);
\draw[color=black] (3.15,3.2) node {$2s_2-2$};
\draw [fill=black] (1.,1.) circle (1.5pt);
\draw[color=black] (0.9,1.2) node {$2n_1+3$};
\draw [fill=black] (3.,1.) circle (1.5pt);
\draw[color=black] (3.15,1.2) node {$2n_1+2$};
\draw [fill=black] (1.,0.) circle (1.5pt);
\draw[color=black] (0.9,-0.2) node {$2n_1+1$};
\draw [fill=black] (3.,0.) circle (1.5pt);
\draw[color=black] (3.15,-0.2) node {$2n_1$};
\draw [fill=black] (1.02,1.58) circle (0.5pt);
\draw [fill=black] (1.,2.) circle (0.5pt);
\draw [fill=black] (1.,2.4) circle (0.5pt);
\draw [fill=black] (3.,1.6) circle (0.5pt);
\draw [fill=black] (3.,2.) circle (0.5pt);
\draw [fill=black] (3.,2.4) circle (0.5pt);
\draw [fill=black] (1.,-2.4) circle (1.5pt);
\draw[color=black] (0.9,-2.2) node {$2s_k-1$};
\draw [fill=black] (1.,-4.4) circle (1.5pt);
\draw[color=black] (0.9,-4.2) node {$2s_{k-1}+3$};
\draw [fill=black] (1.,-5.4) circle (1.5pt);
\draw[color=black] (0.9,-5.6) node {$2s_{k-1}+1$};
\draw [fill=sqsqsq] (1.,-0.8) circle (0.5pt);
\draw [fill=black] (1.,-1.2) circle (0.5pt);
\draw [fill=black] (1.,-1.6) circle (0.5pt);
\draw [fill=black] (1.,-3.) circle (0.5pt);
\draw [fill=black] (1.0066666666666668,-3.4066666666666645) circle (0.5pt);
\draw [fill=black] (1.,-3.8) circle (0.5pt);
\draw [fill=black] (3.006666666666667,-2.4066666666666654) circle (1.5pt);
\draw[color=black] (3.15,-2.18) node {$2s_k-2$};
\draw [fill=black] (3.,-4.4) circle (1.5pt);
\draw[color=black] (3.15,-4.2) node {$2s_{k-1}+2$};
\draw [fill=black] (3.,-5.4) circle (1.5pt);
\draw[color=black] (3.15,-5.6) node {$2s_{k-1}$};
\draw [fill=black] (3.,-3.8) circle (0.5pt);
\draw [fill=black] (3.,-3.4) circle (0.5pt);
\draw [fill=black] (3.,-3.) circle (0.5pt);
\draw (0.4,4.2) node[anchor=north west] {\small{Le poset critique $t_{2n_1}((Q_{2n_2})^{\star})$}};
\draw (-4.2,4.2) node[anchor=north west] {\small{Le poset critique  $Q_{2n_1}$}};

\draw (1.1,-1.4) node[anchor=north west] {\small{Le poset critique $t_{2s_{k-1}}((Q_{2n_k})^{\star})$}};
\end{scriptsize}
\end{tikzpicture}
\caption{\textbf{Le poset (-1)-critique $Q_{2n_1,\ldots, 2n_k}$} (0 est le sommet non critique).}  
\label{poset (-1)-crit Q}
\end{center} 
\end{figure}
%\vspace{2cm}

%%%%%%%%%%%%%%%%%%%%%%%%%%%%%%%%%%%%%%
%%%%%%%%%%%%%%%%%%%%%%%%%%%%%%%%%%%%%%%%%%%%%% $Q'_{2n_1,\ldots,2n_k}$%%%%%%%%
\begin{figure}[h]
\begin{center}
\definecolor{sqsqsq}{rgb}{0.12549019607843137,0.12549019607843137,0.12549019607843137}
\definecolor{ffqqqq}{rgb}{1.,0.,0.}
\definecolor{qqqqff}{rgb}{0.,0.,1.}
\begin{tikzpicture}[line cap=round,line join=round,>=triangle 45,x=1.5cm,y=1.5cm]
%\clip(-8.166666666666668,-6.526666666666663) rectangle (13.886666666666667,4.513333333333331);
\draw [line width=0.4pt] (-4.6,3.4)-- (-3.6,3.4);
\draw [line width=0.4pt] (-3.6,3.4)-- (-3.6,-0.4);
\draw [line width=0.4pt] (-3.6,-0.4)-- (-4.6,-0.4);
\draw [line width=0.4pt] (-4.6,-0.4)-- (-4.6,3.4);
\draw [line width=0.8pt,color=qqqqff] (-2.4,3.4)-- (-1.4,3.4);
\draw [line width=0.8pt,color=qqqqff] (-1.4,3.4)-- (-1.38,-0.39333333333333304);
\draw [line width=0.8pt,color=qqqqff] (-1.38,-0.39333333333333304)-- (-2.4,-0.4);
\draw [line width=0.8pt,color=qqqqff] (-2.4,-0.4)-- (-2.4,3.4);
\draw [line width=0.8pt,color=qqqqff] (0.4,3.4)-- (1.4,3.4);
\draw [line width=0.8pt,color=qqqqff] (1.4,3.4)-- (1.4,-0.4);
\draw [line width=0.8pt,color=qqqqff] (1.4,-0.4)-- (0.4,-0.4);
\draw [line width=0.8pt,color=qqqqff] (0.4,-0.4)-- (0.4,3.4);
\draw [line width=0.4pt] (2.6,3.4)-- (3.6,3.4);
\draw [line width=0.4pt] (3.6,3.4)-- (3.6,-0.4);
\draw [line width=0.4pt] (3.6,-0.4)-- (2.6,-0.4);
\draw [line width=0.4pt] (2.6,-0.4)-- (2.6,3.4);
\draw [line width=0.8pt,color=ffqqqq] (1.,3.)-- (3.,3.);
\draw [line width=0.8pt,color=ffqqqq] (2.07,3.) -- (2.,2.91);
\draw [line width=0.8pt,color=ffqqqq] (2.07,3.) -- (2.,3.09);
\draw [line width=0.8pt,color=ffqqqq] (1.,3.)-- (3.,1.);
\draw [line width=0.8pt,color=ffqqqq] (2.049497474683058,1.9505025253169417) -- (1.9363603896932107,1.9363603896932107);
\draw [line width=0.8pt,color=ffqqqq] (2.049497474683058,1.9505025253169417) -- (2.063639610306788,2.0636396103067893);
\draw [line width=0.8pt,color=ffqqqq] (1.,1.)-- (3.,1.);
\draw [line width=0.8pt,color=ffqqqq] (2.07,1.) -- (2.,0.91);
\draw [line width=0.8pt,color=ffqqqq] (2.07,1.) -- (2.,1.09);
\draw [line width=0.8pt,color=ffqqqq] (1.,1.)-- (3.,0.);
\draw [line width=0.8pt,color=ffqqqq] (2.062609903369994,0.4686950483150028) -- (1.9597507764050035,0.4195015528100071);
\draw [line width=0.8pt,color=ffqqqq] (2.062609903369994,0.4686950483150028) -- (2.040249223594997,0.5804984471899923);
\draw [line width=0.8pt,color=ffqqqq] (1.,0.)-- (3.,0.);
\draw [line width=0.8pt,color=ffqqqq] (2.07,0.) -- (2.,-0.09);
\draw [line width=0.8pt,color=ffqqqq] (2.07,0.) -- (2.,0.09);
\draw [line width=0.8pt,color=ffqqqq] (1.,3.)-- (3.,0.);
\draw [line width=0.8pt,color=ffqqqq] (2.0388290137357656,1.4417564793963507) -- (1.9251154735095939,1.4500769823397293);
\draw [line width=0.8pt,color=ffqqqq] (2.0388290137357656,1.4417564793963507) -- (2.074884526490405,1.5499230176602705);
\draw [line width=0.8pt,color=ffqqqq] (-4.,3.)-- (-2.,3.);
\draw [line width=0.8pt,color=ffqqqq] (-2.93,3.) -- (-3.,2.91);
\draw [line width=0.8pt,color=ffqqqq] (-2.93,3.) -- (-3.,3.09);
\draw [line width=0.8pt,color=ffqqqq] (-4.,3.)-- (-2.,2.);
\draw [line width=0.8pt,color=ffqqqq] (-2.9373900966300064,2.468695048315003) -- (-3.040249223594997,2.419501552810008);
\draw [line width=0.8pt,color=ffqqqq] (-2.9373900966300064,2.468695048315003) -- (-2.9597507764050044,2.5804984471899926);
\draw [line width=0.8pt,color=ffqqqq] (-4.,3.)-- (-2.,0.);
\draw [line width=0.8pt,color=ffqqqq] (-2.961170986264234,1.4417564793963507) -- (-3.074884526490406,1.4500769823397293);
\draw [line width=0.8pt,color=ffqqqq] (-2.961170986264234,1.4417564793963507) -- (-2.9251154735095937,1.5499230176602705);
\draw [line width=0.8pt,color=ffqqqq] (-4.,2.)-- (-2.,2.);
\draw [line width=0.8pt,color=ffqqqq] (-2.93,2.) -- (-3.,1.91);
\draw [line width=0.8pt,color=ffqqqq] (-2.93,2.) -- (-3.,2.09);
\draw [line width=0.8pt,color=ffqqqq] (-4.,2.)-- (-2.,0.);
\draw [line width=0.8pt,color=ffqqqq] (-2.9505025253169417,0.9505025253169419) -- (-3.0636396103067898,0.9363603896932108);
\draw [line width=0.8pt,color=ffqqqq] (-2.9505025253169417,0.9505025253169419) -- (-2.9363603896932107,1.0636396103067896);
\draw [line width=0.8pt,color=ffqqqq] (-4.,0.)-- (-2.,0.);
\draw [line width=0.8pt,color=ffqqqq] (-2.93,0.) -- (-3.,-0.09);
\draw [line width=0.8pt,color=ffqqqq] (-2.93,0.) -- (-3.,0.09);
\draw [line width=1.2pt,color=qqqqff] (0.4,1.0333333333333328)-- (-1.3874518330913919,1.0200310096673566);
\draw [line width=1.2pt,color=qqqqff] (-0.5737237012674723,1.026086823192083) -- (-0.49447010193102264,1.1266794024025664);
\draw [line width=1.2pt,color=qqqqff] (-0.5737237012674723,1.026086823192083) -- (-0.49298173116036803,0.9266849405981227);
\draw [line width=0.8pt,color=qqqqff] (0.4,-2.)-- (1.4,-2.);
\draw [line width=0.8pt,color=qqqqff] (1.4,-2.)-- (1.4,-5.8);
\draw [line width=0.8pt,color=qqqqff] (1.4,-5.8)-- (0.4,-5.8);
\draw [line width=0.8pt,color=qqqqff] (0.4,-5.8)-- (0.4,-2.);
\draw [line width=1.2pt,color=qqqqff] (0.7933333333333333,-2.)-- (-1.387451833091392,1.0200310096673566);
\draw [line width=1.2pt,color=qqqqff] (-0.34389357758033995,-0.42512660629560567) -- (-0.21598688879063352,-0.4314415855396833);
\draw [line width=1.2pt,color=qqqqff] (-0.34389357758033995,-0.42512660629560567) -- (-0.37813161096742537,-0.5485274047929606);
\draw [line width=0.4pt] (2.6,-2.)-- (3.6,-2.);
\draw [line width=0.4pt] (3.6,-2.)-- (3.6,-5.8);
\draw [line width=0.4pt] (3.6,-5.8)-- (2.6,-5.8);
\draw [line width=0.4pt] (2.6,-5.8)-- (2.6,-2.);
\draw [line width=0.8pt,color=ffqqqq] (1.,-2.4)-- (3.006666666666667,-2.4066666666666654);
\draw [line width=0.8pt,color=ffqqqq] (2.0733329470273296,-2.403565890189458) -- (2.0030343316611714,-2.493332836654184);
\draw [line width=0.8pt,color=ffqqqq] (2.0733329470273296,-2.403565890189458) -- (2.0036323350054954,-2.3133338300124797);
\draw [line width=0.8pt,color=ffqqqq] (1.,-2.4)-- (3.,-4.4);
\draw [line width=0.8pt,color=ffqqqq] (2.049497474683058,-3.4494974746830587) -- (1.9363603896932107,-3.463639610306789);
\draw [line width=0.8pt,color=ffqqqq] (2.049497474683058,-3.4494974746830587) -- (2.063639610306788,-3.3363603896932115);
\draw [line width=0.8pt,color=ffqqqq] (1.,-2.4)-- (3.,-5.4);
\draw [line width=0.8pt,color=ffqqqq] (2.0388290137357656,-3.95824352060365) -- (1.9251154735095939,-3.949923017660272);
\draw [line width=0.8pt,color=ffqqqq] (2.0388290137357656,-3.95824352060365) -- (2.074884526490405,-3.850076982339731);
\draw [line width=0.8pt,color=ffqqqq] (1.,-4.4)-- (3.,-4.4);
\draw [line width=0.8pt,color=ffqqqq] (2.07,-4.4) -- (2.,-4.49);
\draw [line width=0.8pt,color=ffqqqq] (2.07,-4.4) -- (2.,-4.31);
\draw [line width=0.8pt,color=ffqqqq] (1.,-4.4)-- (3.,-5.4);
\draw [line width=0.8pt,color=ffqqqq] (2.062609903369994,-4.931304951684997) -- (1.9597507764050035,-4.980498447189992);
\draw [line width=0.8pt,color=ffqqqq] (2.062609903369994,-4.931304951684997) -- (2.040249223594997,-4.819501552810008);
\draw [line width=0.8pt,color=ffqqqq] (1.,-5.4)-- (3.,-5.4);
\draw [line width=0.8pt,color=ffqqqq] (2.07,-5.4) -- (2.,-5.49);
\draw [line width=0.8pt,color=ffqqqq] (2.07,-5.4) -- (2.,-5.31);
\draw [line width=0.8pt,color=ffqqqq] (-2.,3.)-- (-2.,2.);
\draw [line width=0.8pt,color=ffqqqq] (-2.,2.43) -- (-2.09,2.5);
\draw [line width=0.8pt,color=ffqqqq] (-2.,2.43) -- (-1.91,2.5);
\draw [line width=0.8pt,color=ffqqqq] (-2.,2.)-- (-2.,1.6);
\draw [line width=0.8pt,color=ffqqqq] (-2.,1.73) -- (-2.09,1.8);
\draw [line width=0.8pt,color=ffqqqq] (-2.,1.73) -- (-1.91,1.8);
\draw [line width=0.8pt,color=ffqqqq] (-2.,0.6)-- (-2.,0.);
\draw [line width=0.8pt,color=ffqqqq] (-2.,0.23) -- (-2.09,0.3);
\draw [line width=0.8pt,color=ffqqqq] (-2.,0.23) -- (-1.91,0.3);
\begin{scriptsize}
\draw [fill=black] (-4.,3.) circle (2.5pt);
\draw[color=black] (-4.2,3.1) node {$0$};
\draw [fill=black] (-2.,3.) circle (1.5pt);
\draw[color=black] (-1.8,3.1) node {$1$};
\draw [fill=black] (-4.,2.) circle (1.5pt);
\draw[color=black] (-4.2,2.1) node {$2$};
\draw [fill=black] (-2.,2.) circle (1.0pt);
\draw[color=black] (-1.8,2.1) node {$3$};
\draw [fill=black] (-4.,0.) circle (1.5pt);
\draw[color=black] (-4.2,-0.2) node {$2n_1-2$};
\draw [fill=black] (-2.,0.) circle (1.5pt);
\draw[color=black] (-1.8,-0.2) node {$2n_1-1$};
\draw [fill=black] (-4.,1.4) circle (0.5pt);
\draw [fill=black] (-4.,1.) circle (0.5pt);
\draw [fill=black] (-4.,0.6) circle (0.5pt);
\draw [fill=black] (-2.,1.1) circle (0.5pt);
\draw [fill=black] (-2.,0.8) circle (0.5pt);
\draw [fill=black] (-2.,1.4) circle (0.5pt);
\draw [fill=black] (-0.4,0.1) circle (0.5pt);
\draw [fill=black] (-0.4,0.3) circle (0.5pt);
\draw [fill=black] (-0.4,0.5) circle (0.5pt);
\draw [fill=black] (1.,3.) circle (1.5pt);
\draw[color=black] (0.9,3.15) node {$2s_2-1$};
\draw [fill=black] (3.,3.) circle (1.5pt);
\draw[color=black] (3.15,3.2) node {$2s_2-2$};
\draw [fill=black] (1.,1.) circle (1.5pt);
\draw[color=black] (0.9,1.2) node {$2n_1+3$};
\draw [fill=black] (3.,1.) circle (1.5pt);
\draw[color=black] (3.15,1.2) node {$2n_1+2$};
\draw [fill=black] (1.,0.) circle (1.5pt);
\draw[color=black] (0.9,-0.2) node {$2n_1+1$};
\draw [fill=black] (3.,0.) circle (1.5pt);
\draw[color=black] (3.15,-0.2) node {$2n_1$};
\draw [fill=black] (1.02,1.58) circle (0.5pt);
\draw [fill=black] (1.,2.) circle (0.5pt);
\draw [fill=black] (1.,2.4) circle (0.5pt);
\draw [fill=black] (3.,1.6) circle (0.5pt);
\draw [fill=black] (3.,2.) circle (0.5pt);
\draw [fill=black] (3.,2.4) circle (0.5pt);
\draw [fill=black] (1.,-2.4) circle (1.5pt);
\draw[color=black] (0.9,-2.2) node {$2s_k-1$};
\draw [fill=black] (1.,-4.4) circle (1.5pt);
\draw[color=black] (0.9,-4.2) node {$2s_{k-1}+3$};
\draw [fill=black] (1.,-5.4) circle (1.5pt);
\draw[color=black] (0.9,-5.6) node {$2s_{k-1}+1$};
\draw [fill=sqsqsq] (1.,-0.8) circle (0.5pt);
\draw [fill=black] (1.,-1.2) circle (0.5pt);
\draw [fill=black] (1.,-1.6) circle (0.5pt);
\draw [fill=black] (1.,-3.) circle (0.5pt);
\draw [fill=black] (1.0066666666666668,-3.4066666666666645) circle (0.5pt);
\draw [fill=black] (1.,-3.8) circle (0.5pt);
\draw [fill=black] (3.006666666666667,-2.4066666666666654) circle (1.5pt);
\draw[color=black] (3.15,-2.18) node {$2s_k-2$};
\draw [fill=black] (3.,-4.4) circle (1.5pt);
\draw[color=black] (3.15,-4.2) node {$2s_{k-1}+2$};
\draw [fill=black] (3.,-5.4) circle (1.5pt);
\draw[color=black] (3.15,-5.6) node {$2s_{k-1}$};
\draw [fill=black] (3.,-3.8) circle (0.5pt);
\draw [fill=black] (3.,-3.4) circle (0.5pt);
\draw [fill=black] (3.,-3.) circle (0.5pt);
\draw (0.4,4.2) node[anchor=north west] {\small{Le poset critique $t_{2n_1}((Q_{2n_2})^{\star})$}};
\draw (-4.2,4.2) node[anchor=north west] {\small{Le poset  $Q'_{2n_1}$}};

\draw (1.1,-1.4) node[anchor=north west] {\small{Le poset critique $t_{2s_{k-1}}((Q_{2n_k})^{\star})$}};

\end{scriptsize}
\end{tikzpicture}
\caption{\textbf{Le poset (-1)-critique $Q'_{2n_1,\ldots, 2n_k}$} (0 est le sommet non critique).}  
\label{poset (-1)-crit Q'}
\end{center} 
\end{figure}
%%%%%%%%%%%%%%%%%%%%%%%%%%%%%%%%%%%%%%%
%%%%%%%%%%%%%%%%%%%%%%% le poset (-1)-critique H_{1,2n_1,2n_2}$  %%%%%%%%%%%
\begin{figure}[h]
\begin{center}
\definecolor{ffqqqq}{rgb}{1.,0.,0.}
\definecolor{qqqqff}{rgb}{0.,0.,1.}
\begin{tikzpicture}[line cap=round,line join=round,>=triangle 45,x=1.5cm,y=1.5cm]
%\clip(-8.166666666666668,-6.526666666666663) rectangle (13.886666666666667,4.513333333333331);
\draw [line width=0.4pt] (-4.6,3.4)-- (-3.6,3.4);
\draw [line width=0.4pt] (-3.6,3.4)-- (-3.6,-0.4);
\draw [line width=0.4pt] (-3.6,-0.4)-- (-4.6,-0.4);
\draw [line width=0.4pt] (-4.6,-0.4)-- (-4.6,3.4);
\draw [line width=0.8pt,color=qqqqff] (-2.4,3.4)-- (-1.4,3.4);
\draw [line width=0.8pt,color=qqqqff] (-1.4,3.4)-- (-1.38,-0.39333333333333304);
\draw [line width=0.8pt,color=qqqqff] (-1.38,-0.39333333333333304)-- (-2.4,-0.4);
\draw [line width=0.8pt,color=qqqqff] (-2.4,-0.4)-- (-2.4,3.4);
\draw [line width=0.8pt,color=qqqqff] (1.4,3.4)-- (2.4,3.4);
\draw [line width=0.8pt,color=qqqqff] (2.4,3.4)-- (2.4,-0.4);
\draw [line width=0.8pt,color=qqqqff] (2.4,-0.4)-- (1.4,-0.4);
\draw [line width=0.8pt,color=qqqqff] (1.4,-0.4)-- (1.4,3.4);
\draw [line width=0.4pt] (3.6,3.4)-- (4.6,3.4);
\draw [line width=0.4pt] (4.6,3.4)-- (4.6,-0.4);
\draw [line width=0.4pt] (4.6,-0.4)-- (3.6,-0.4);
\draw [line width=0.4pt] (3.6,-0.4)-- (3.6,3.4);
\draw [line width=0.8pt,color=ffqqqq] (2.,3.)-- (4.,3.);
\draw [line width=0.8pt,color=ffqqqq] (3.07,3.) -- (3.,2.91);
\draw [line width=0.8pt,color=ffqqqq] (3.07,3.) -- (3.,3.09);
\draw [line width=0.8pt,color=ffqqqq] (2.,3.)-- (4.,1.);
\draw [line width=0.8pt,color=ffqqqq] (3.049497474683058,1.9505025253169417) -- (2.9363603896932107,1.9363603896932107);
\draw [line width=0.8pt,color=ffqqqq] (3.049497474683058,1.9505025253169417) -- (3.063639610306788,2.0636396103067893);
\draw [line width=0.8pt,color=ffqqqq] (2.,1.)-- (4.,1.);
\draw [line width=0.8pt,color=ffqqqq] (3.07,1.) -- (3.,0.91);
\draw [line width=0.8pt,color=ffqqqq] (3.07,1.) -- (3.,1.09);
\draw [line width=0.8pt,color=ffqqqq] (2.,1.)-- (4.,0.);
\draw [line width=0.8pt,color=ffqqqq] (3.062609903369994,0.4686950483150028) -- (2.9597507764050035,0.4195015528100071);
\draw [line width=0.8pt,color=ffqqqq] (3.062609903369994,0.4686950483150028) -- (3.040249223594997,0.5804984471899923);
\draw [line width=0.8pt,color=ffqqqq] (2.,0.)-- (4.,0.);
\draw [line width=0.8pt,color=ffqqqq] (3.07,0.) -- (3.,-0.09);
\draw [line width=0.8pt,color=ffqqqq] (3.07,0.) -- (3.,0.09);
\draw [line width=0.8pt,color=ffqqqq] (2.,3.)-- (4.,0.);
\draw [line width=0.8pt,color=ffqqqq] (3.0388290137357656,1.4417564793963507) -- (2.9251154735095937,1.4500769823397293);
\draw [line width=0.8pt,color=ffqqqq] (3.0388290137357656,1.4417564793963507) -- (3.074884526490405,1.5499230176602705);
\draw [line width=0.8pt,color=ffqqqq] (-4.,3.)-- (-2.,3.);
\draw [line width=0.8pt,color=ffqqqq] (-2.93,3.) -- (-3.,2.91);
\draw [line width=0.8pt,color=ffqqqq] (-2.93,3.) -- (-3.,3.09);
\draw [line width=0.8pt,color=ffqqqq] (-4.,3.)-- (-2.,2.);
\draw [line width=0.8pt,color=ffqqqq] (-2.9373900966300064,2.468695048315003) -- (-3.040249223594997,2.419501552810008);
\draw [line width=0.8pt,color=ffqqqq] (-2.9373900966300064,2.468695048315003) -- (-2.9597507764050044,2.5804984471899926);
\draw [line width=0.8pt,color=ffqqqq] (-4.,3.)-- (-2.,0.);
\draw [line width=0.8pt,color=ffqqqq] (-2.961170986264234,1.4417564793963507) -- (-3.074884526490406,1.4500769823397293);
\draw [line width=0.8pt,color=ffqqqq] (-2.961170986264234,1.4417564793963507) -- (-2.9251154735095937,1.5499230176602705);
\draw [line width=0.8pt,color=ffqqqq] (-4.,2.)-- (-2.,2.);
\draw [line width=0.8pt,color=ffqqqq] (-2.93,2.) -- (-3.,1.91);
\draw [line width=0.8pt,color=ffqqqq] (-2.93,2.) -- (-3.,2.09);
\draw [line width=0.8pt,color=ffqqqq] (-4.,2.)-- (-2.,0.);
\draw [line width=0.8pt,color=ffqqqq] (-2.9505025253169417,0.9505025253169419) -- (-3.0636396103067898,0.9363603896932108);
\draw [line width=0.8pt,color=ffqqqq] (-2.9505025253169417,0.9505025253169419) -- (-2.9363603896932107,1.0636396103067896);
\draw [line width=0.8pt,color=ffqqqq] (-4.,0.)-- (-2.,0.);
\draw [line width=0.8pt,color=ffqqqq] (-2.93,0.) -- (-3.,-0.09);
\draw [line width=0.8pt,color=ffqqqq] (-2.93,0.) -- (-3.,0.09);
\draw [line width=1.2pt,color=qqqqff] (1.4,1.0333333333333328)-- (-1.3874518330913919,1.0200310096673566);
\draw [line width=1.2pt,color=qqqqff] (-0.07372500559923689,1.0263003985265704) -- (0.00579686723708619,1.1266810328172723);
\draw [line width=1.2pt,color=qqqqff] (-0.07372500559923689,1.0263003985265704) -- (0.006751299671523157,0.9266833101834181);
\draw [line width=0.8pt,color=ffqqqq] (-2.,3.)-- (-2.,2.);
\draw [line width=0.8pt,color=ffqqqq] (-2.,2.43) -- (-2.09,2.5);
\draw [line width=0.8pt,color=ffqqqq] (-2.,2.43) -- (-1.91,2.5);
\draw [line width=0.8pt,color=ffqqqq] (-2.,2.)-- (-2.,1.6);
\draw [line width=0.8pt,color=ffqqqq] (-2.,1.73) -- (-2.09,1.8);
\draw [line width=0.8pt,color=ffqqqq] (-2.,1.73) -- (-1.91,1.8);
\draw [line width=0.8pt,color=ffqqqq] (-2.,0.6)-- (-2.,0.);
\draw [line width=0.8pt,color=ffqqqq] (-2.,0.23) -- (-2.09,0.3);
\draw [line width=0.8pt,color=ffqqqq] (-2.,0.23) -- (-1.91,0.3);
\draw [line width=1.2pt,color=qqqqff] (1.913333333333332,-0.4)-- (0.,-1.38);
\draw [line width=1.2pt,color=qqqqff] (0.885463197477972,-0.9264700695844524) -- (0.9110790796860996,-0.8009956635141334);
\draw [line width=1.2pt,color=qqqqff] (0.885463197477972,-0.9264700695844524) -- (1.0022542536472308,-0.9790043364858663);
\draw [line width=1.2pt,color=qqqqff] (0.,-1.38)-- (-1.8865578812473318,-0.39664416915847917);
\draw [line width=1.2pt,color=qqqqff] (-1.0142201455585458,-0.8513444502256549) -- (-0.8970568976816852,-0.7996455784106399);
\draw [line width=1.2pt,color=qqqqff] (-1.0142201455585458,-0.8513444502256549) -- (-0.9895009835656478,-0.9769985907478389);
\draw [line width=0.8pt,color=ffqqqq] (2.,1.)-- (2.,0.);
\draw [line width=0.8pt,color=ffqqqq] (2.,0.43) -- (1.91,0.5);
\draw [line width=0.8pt,color=ffqqqq] (2.,0.43) -- (2.09,0.5);
\draw [line width=0.8pt,color=ffqqqq] (2.,3.)-- (2.,2.4);
\draw [line width=0.8pt,color=ffqqqq] (2.,2.63) -- (1.91,2.7);
\draw [line width=0.8pt,color=ffqqqq] (2.,2.63) -- (2.09,2.7);
\draw [line width=0.8pt,color=ffqqqq] (2.,1.4)-- (2.,1.);
\draw [line width=0.8pt,color=ffqqqq] (2.,1.13) -- (1.91,1.2);
\draw [line width=0.8pt,color=ffqqqq] (2.,1.13) -- (2.09,1.2);
\begin{scriptsize}
\draw [fill=black] (-4.,3.) circle (1.5pt);
\draw[color=black] (-4.2,3.1) node {$1$};
\draw [fill=black] (-2.,3.) circle (1.5pt);
\draw[color=black] (-1.8,3.1) node {$2$};
\draw [fill=black] (-4.,2.) circle (1.5pt);
\draw[color=black] (-4.2,2.1) node {$3$};
\draw [fill=black] (-2.,2.) circle (1.0pt);
\draw[color=black] (-1.8,2.1) node {$4$};
\draw [fill=black] (-4.,0.) circle (1.5pt);
\draw[color=black] (-4.2,-0.2) node {$2n_1-1$};
\draw [fill=black] (-2.,0.) circle (1.5pt);
\draw[color=black] (-1.7,-0.1) node {$2n_1$};
\draw [fill=black] (-4.,1.4) circle (0.5pt);
\draw [fill=black] (-4.,1.) circle (0.5pt);
\draw [fill=black] (-4.,0.6) circle (0.5pt);
\draw [fill=black] (-2.,1.1) circle (0.5pt);
\draw [fill=black] (-2.,0.8) circle (0.5pt);
\draw [fill=black] (-2.,1.4) circle (0.5pt);
\draw [fill=black] (2.,3.) circle (1.5pt);
\draw[color=black] (1.8,3.2) node {$2s_2$};
\draw [fill=black] (4.,3.) circle (1.5pt);
\draw[color=black] (4.2,3.2) node {$2s_2-1$};
\draw [fill=black] (2.,1.) circle (1.5pt);
\draw[color=black] (1.8,1.2) node {$2n_1+4$};
\draw [fill=black] (4.,1.) circle (1.5pt);
\draw[color=black] (4.2,1.2) node {$2n_1+3$};
\draw [fill=black] (2.,0.) circle (1.5pt);
\draw[color=black] (1.8,-0.2) node {$2n_1+2$};
\draw [fill=black] (4.,0.) circle (1.5pt);
\draw[color=black] (4.2,-0.2) node {$2n_1+1$};
\draw [fill=black] (2.,2.2) circle (0.5pt);
\draw [fill=black] (4.,1.6) circle (0.5pt);
\draw [fill=black] (4.,2.) circle (0.5pt);
\draw [fill=black] (4.,2.4) circle (0.5pt);
\draw [fill=black] (0.,-1.38) circle (2.5pt);
\draw[color=black] (0.1,-1.1533333333333327) node {$0$};
\draw [fill=black] (2.,1.9) circle (0.5pt);
\draw [fill=black] (2.,1.6) circle (0.5pt);
\draw (-4.2,4.2) node[anchor=north west] {\small{Le poset $t_{1}(Q'_{2n_1})$}};
\draw (1.8,4.2) node[anchor=north west] {\small{Le poset  $t_{2n_1+1}((Q'_{2n_2})^{\star})$}};
\end{scriptsize}
\end{tikzpicture}
\caption{\textbf{Le poset (-1)-critique $R_{1,2n_1,2n_2}$} (0 est le sommet non critique).}  
\label{poset (-1)-crit H}
\end{center} 
\end{figure}
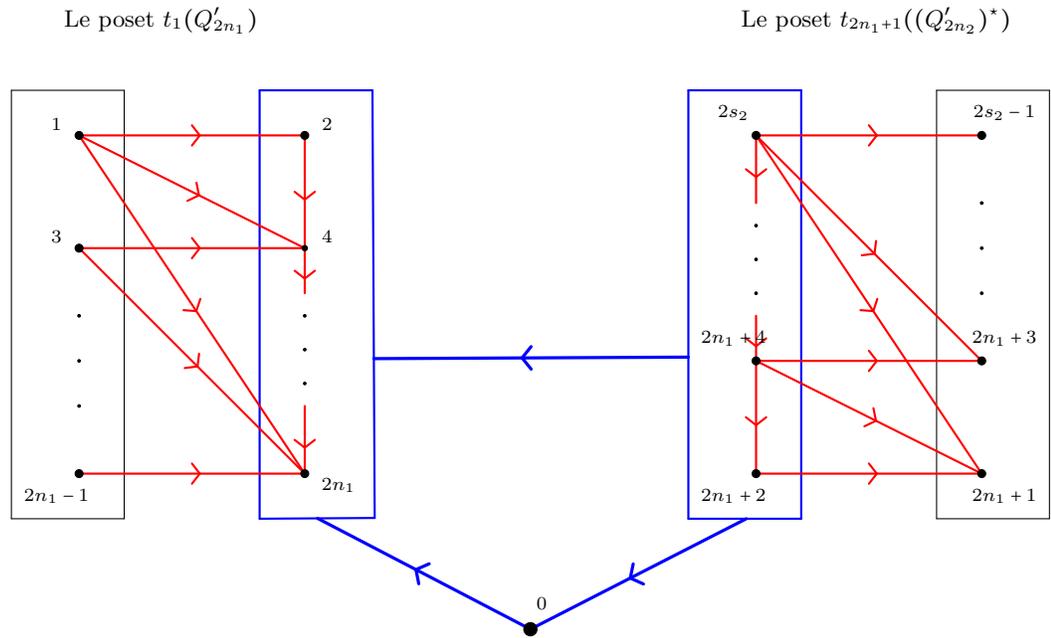
 
 Comme nous l'avons vu dans l'introduction, notre caractérisation des posets (-1)-critiques se base sur la reconnaissance des graphes de comparabilité de la la classe des graphes (-1)-critiques. Les propriétés et les structures des graphes de comparabilité et de leurs orientations ont été étudiées par de nombreux auteurs. Nous citons par exemples, Gallai \cite{strong interval}, Golumbic \cite{Golumbic.comparability graphs}, S. Földes et P.L. Hammer 
\cite{Hammer-Foldes}, Gilmore et Hoffman \cite{caract.comparabity graphs }, F. Maffray et M. Preissmann \cite{strong intervals}. 

Dans \cite{Hammer-Foldes}, les auteurs ont donné une caractérisation des graphes scindés qui sont de comparabilité. les graphes $G_1$, $G_2$, $G_3$ et $G_4$ de la  figure~\ref{les graphes de Hammer} interviennent dans cette caractérisation. Notons que, à isomorphisme près,\,   $G_2=\overline{G_1}$ et $G_4=\overline{G_3}$, et que seul le graphe $G_3$ est de comparabilité.

\begin{figure}[h]
\begin{center}
\begin{tikzpicture}[line cap=round,line join=round,>=triangle 45,x=0.75cm,y=0.75cm]
%\clip(-12.68,-7.56) rectangle (20.4,9.);
\draw [line width=0.8pt] (-5.,6.)-- (-5.,4.);
\draw [line width=0.8pt] (-5.,4.)-- (-6.,2.);
\draw [line width=0.8pt] (-6.,2.)-- (-4.,2.);
\draw [line width=0.8pt] (-4.,2.)-- (-5.,4.);
\draw [line width=0.8pt] (-6.,2.)-- (-6.,0.);
\draw [line width=0.8pt] (-4.,2.)-- (-4.,0.);
\draw [line width=0.8pt] (0.,4.)-- (-1.,2.);
\draw [line width=0.8pt] (-1.,2.)-- (1.,2.);
\draw [line width=0.8pt] (1.,2.)-- (0.,4.);
\draw [line width=0.8pt] (-1.,2.)-- (-2.,0.);
\draw [line width=0.8pt] (-2.,0.)-- (0.,0.);
\draw [line width=0.8pt] (0.,0.)-- (-1.,2.);
\draw [line width=0.8pt] (1.,2.)-- (0.,0.);
\draw [line width=0.8pt] (0.,0.)-- (2.,0.);
\draw [line width=0.8pt] (2.,0.)-- (1.,2.);
\draw [line width=0.8pt] (4.,1.)-- (6.,0.);
\draw [line width=0.8pt] (6.,0.)-- (6.,2.);
\draw [line width=0.8pt] (6.,2.)-- (4.,1.);
\draw [line width=0.8pt] (6.,2.)-- (8.,2.);
\draw [line width=0.8pt] (8.,2.)-- (8.,0.);
\draw [line width=0.8pt] (8.,0.)-- (6.,0.);
\draw [line width=0.8pt] (6.,2.)-- (8.,0.);
\draw [line width=0.8pt] (8.,2.)-- (6.,0.);
\draw [line width=0.8pt] (7.,4.)-- (6.,2.);
\draw [line width=0.8pt] (7.,4.)-- (8.,2.);
\draw [line width=0.8pt] (8.,2.)-- (10.,1.);
\draw [line width=0.8pt] (10.,1.)-- (8.,0.);
\draw [line width=0.8pt] (13.,2.)-- (12.,0.);
\draw [line width=0.8pt] (12.,0.)-- (14.,0.);
\draw [line width=0.8pt] (14.,0.)-- (13.,2.);
\draw [line width=0.8pt] (13.,2.)-- (15.,2.);
\draw [line width=0.8pt] (15.,2.)-- (14.,0.);
\draw [line width=0.8pt] (14.,0.)-- (16.,0.);
\draw [line width=0.8pt] (16.,0.)-- (15.,2.);
\draw [line width=0.8pt] (15.,2.)-- (15.,4.);
\draw [line width=0.8pt] (13.,2.)-- (13.,4.);
%\begin{scriptsize}
\draw [fill=black] (-6.,2.) circle (2.0pt);
%\draw[color=black] (-5.86,2.25) node {$A$};
\draw [fill=black] (-4.,2.) circle (2.0pt);
%\draw[color=black] (-3.86,2.25) node {$B$};
\draw [fill=black] (-5.,4.) circle (2.0pt);
%\draw[color=black] (-4.86,4.25) node {$C$};
\draw [fill=black] (-5.,6.) circle (2.0pt);
%\draw[color=black] (-4.86,6.25) node {$D$};
\draw [fill=black] (-6.,0.) circle (2.0pt);
%\draw[color=black] (-5.86,0.25) node {$E$};
\draw [fill=black] (-4.,0.) circle (2.0pt);
%\draw[color=black] (-3.86,0.25) node {$F$};
\draw [fill=black] (-2.,0.) circle (2.0pt);
%\draw[color=black] (-1.86,0.25) node {$G$};
\draw [fill=black] (2.,0.) circle (2.0pt);
%\draw[color=black] (2.14,0.25) node {$I$};
\draw [fill=black] (1.,2.) circle (2.0pt);
%\draw[color=black] (1.14,2.25) node {$J$};
\draw [fill=black] (-1.,2.) circle (2.0pt);
%\draw[color=black] (-0.86,2.25) node {$K$};
\draw [fill=black] (0.,4.) circle (2.0pt);
%\draw[color=black] (0.14,4.25) node {$L$};
\draw [fill=black] (0.,0.) circle (2.0pt);
%\draw[color=black] (0.14,0.25) node {$H$};
\draw [fill=black] (4.,1.) circle (2.0pt);
%\draw[color=black] (4.14,1.25) node {$M$};
\draw [fill=black] (6.,2.) circle (2.0pt);
%\draw[color=black] (6.14,2.25) node {$N$};
\draw [fill=black] (6.,0.) circle (2.0pt);
%\draw[color=black] (6.14,0.25) node {$O$};
\draw [fill=black] (8.,0.) circle (2.0pt);
%\draw[color=black] (8.14,0.25) node {$P$};
\draw [fill=black] (8.,2.) circle (2.0pt);
%\draw[color=black] (8.14,2.25) node {$Q$};
\draw [fill=black] (7.,4.) circle (2.0pt);
%\draw[color=black] (7.14,4.25) node {$R$};
\draw [fill=black] (10.,1.) circle (2.0pt);
%\draw[color=black] (10.14,1.25) node {$S$};

\draw[color=black] (-5,-0.95) node {$G_1$};
\draw[color=black] (0,-0.95) node {$G_2$};
\draw[color=black] (7,-0.95) node {$G_3$};
\draw[color=black] (14,-0.95) node {$G_4$};

\draw [fill=black] (12.,0.) circle (2.0pt);
\draw [fill=black] (14.,0.) circle (2.0pt);
\draw [fill=black] (13.,2.) circle (2.0pt);
\draw [fill=black] (16.,0.) circle (2pt);
\draw [fill=black] (15.,2.) circle (2.pt);
\draw [fill=black] (13.,4.) circle (2.pt);
\draw [fill=black] (15.,4.) circle (2.pt);
%\end{scriptsize}
\end{tikzpicture}
\caption{\textbf{Les graphes $G_1, G_2, G_3, G_4$.}}  
\label{les graphes de Hammer}
\end{center} 
\end{figure}
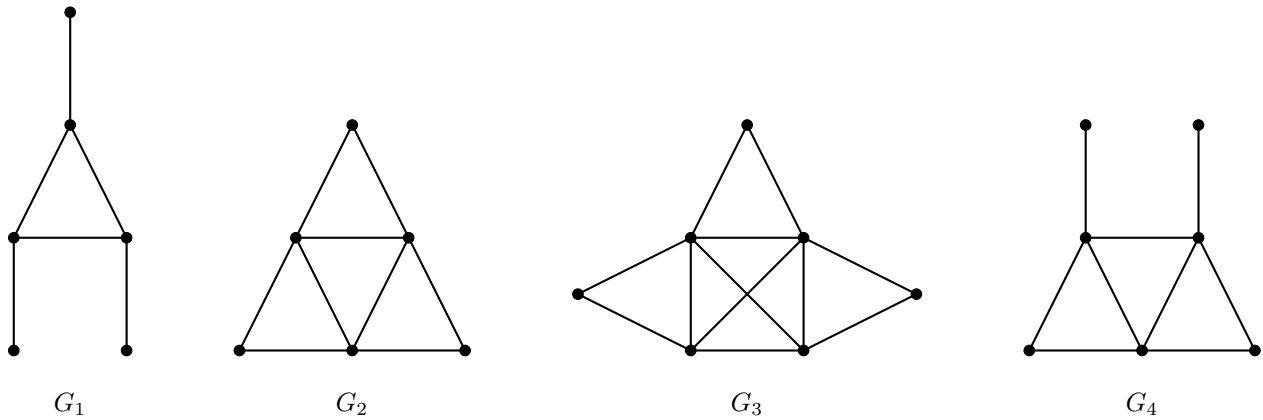
Nous rappelons les deux résultats suivants.
 \begin{theorem}\cite{Hammer-Foldes}\label{Hammer theo 1}
 Un graphe scindé $G$ est de comparabilité si et seulement s'il n'abrite pas $G_1, G_2$ et $G_4$.
 \end{theorem}
\begin{theorem}\cite{Hammer-Foldes}\label{Hammer theo 2} \'Etant donné un graphe scindé $G$, le graphe complémentaire $\overline{G}$ de $G$ est de comparabilité si et seulement si $G$ n'abrite pas $G_1, G_2$ et $G_3$.
\end{theorem}
Remarquons que pour tout entier   $k\geqslant 3$,  $H_{1, 2n_1, \ldots, 2n_k}[\{ 1,2,2n_1+1, 2n_1+2, 2s_2+1,2s_2+2\}]\simeq G_1$. Du Théorème~\ref{Hammer theo 1}, on en déduit que pour tout $k\geqslant 3$, le graphe scindé  $H_{1, 2n_1, \ldots, 2n_k}$ n'est pas de comparabilité.  De plus, comme $\overline{G_1}\simeq G_2$, alors, pour tout $k\geqslant 3$, le graphe $\overline{H_{1, 2n_1, \ldots, 2n_k}}$ abrite $G_2$, et donc d'après le théorème~\ref{Hammer theo 2}, $\overline{H_{1, 2n_1, \ldots, 2n_k}}$ n'est pas un graphe de comparabilité.
 Nous obtenons alors: 

\begin{lemme} \label{k ne depasse pas 3 pour H}
Pour tout entier $ k\geqslant 3$ et pour tous entiers strictement positif $n_1$ et $n_2$, le graphe $H_{1, 2n_1, \ldots, 2n_k}$ et  son complémentaire 
  $ \overline{H_{1, 2n_1, \ldots, 2n_k}}$ ne sont pas des graphes de comparabilité.
 \end{lemme} 
  Du Lemme~\ref{k ne depasse pas 3 pour H},  nous avons qu'a traiter, le cas où $k=2$ pour le graphe \\
  (-1)-critique $H_{1, 2n_1, \ldots, 2n_k}$ et son complémentaire.  Clairement, nous avons le résultat suivant.
  \begin{Fait}\label{R est un poset -1critique}
   Pour tous entiers strictement positifs $n_1$ et $n_2$, le digraphe $R_{1,2n_1,2n_2}$ est une orientation transitive du graphe  $H_{1,2n_1,2n_2}$. Le graphe (-1)-critique $H_{1,2n_1,2n_2}$ est alors un graphe de comparabilité. 
   \end{Fait}
 \begin{corollaire}\label{complément de H} Pour tous entiers strictement positifs  $n_1$ et $n_2$, le graphe $\overline{H_{1,2n_1,2n_2}}$ est de comparabilité si et seulement si le graphe  $H_{1,2n_1,2n_2}$ n'abrite pas $G_3$.
\end{corollaire}
 \begin{proof}[\textbf{Preuve}]
 Le graphe $G=H_{1,2n_1,2n_2}$ étant scindé, d'après le théorème~\ref{Hammer theo 2}, il suffit de montrer la condition suffisante du corollaire. Par hypothèse, $G$ n'abrite pas $G_3$. Comme de plus $G$ est un graphe de comparabilité, alors, d'après  le théorème~\ref{Hammer theo 1}, $G$ n'abrite pas $G_1$ et $G_2$. Il s'ensuit du théorème~\ref{Hammer theo 2}, que le graphe $\overline{G}=\overline{H_{1,2n_1,2n_2}}$ est de comparabilité comme voulu. 
\end{proof}

 Nous proposons dans ce que suit une orientation  transitive de   $\overline{H_{1,2n_1,2n_2}}$, pour tout entiers strictement positifs $n_1$ et 
 $ n_2$. Pour ce faire, nous considérons d'abord, pour tout $n\geqslant 1$, le digraphe transitif $R'_{2n}$ défini sur $\{0\ldots,2n-1\}$ par $A(R'_{2n})= \{ (2i,2j), (2i+1,2j): 0\leqslant i< j\leqslant n-1\}$.
 
  Maintenant, introduisons le digraphe  $R'_{1, 2n_1,  2n_2}$  défini sur $\{0, \ldots, 2s_2\}$ par $A(R'_{1, 2n_1, 2n_2}) =  A(t_1(R'_{2n_1}))\cup A(t_{2n_1+1}((R'_{2n_2})^{\star}))\cup 
\{(2q-1,p): 0 \leqslant p \leqslant 2 n_1, n_1+1\leqslant q \leqslant~ s_2\} 
 \cup  \{(2q,2p-1): 1 \leqslant p \leqslant n_1 \ \text{et} \  n_1+1\leqslant q \leqslant s_2 \ \text{ou} \ q=0\}$ (voir Figure~\ref{poset (-1)-crit R'}). 

\begin{figure}[h]
\begin{center}
\definecolor{qqqqff}{rgb}{0.,0.,1.}
\definecolor{ffqqqq}{rgb}{1.,0.,0.}
\begin{tikzpicture}[line cap=round,line join=round,>=triangle 45,x=1.2cm,y=1.0cm]
%\clip(-11.38,-6.98) rectangle (21.7,9.58);
\draw [line width=0.4pt] (-5.6,5.6)-- (-4.4,5.6);
\draw [line width=0.4pt] (-4.4,5.6)-- (-4.4,-0.6);
\draw [line width=0.4pt] (-5.6,5.6)-- (-5.6,-0.6);
\draw [line width=0.4pt] (-5.6,-0.6)-- (-4.4,-0.6);
\draw [line width=0.4pt] (-2.6,5.6)-- (-1.4,5.6);
\draw [line width=0.4pt] (-1.4,5.6)-- (-1.4,-0.6);
\draw [line width=0.4pt] (-1.4,-0.6)-- (-2.6,-0.6);
\draw [line width=0.4pt] (-2.6,-0.6)-- (-2.6,5.6);
\draw [line width=0.8pt,color=ffqqqq] (-2.,5.)-- (-5.,4.);
\draw [line width=0.8pt,color=ffqqqq] (-3.5996117462953032,4.466796084568232) -- (-3.542690748412273,4.628072245236819);
\draw [line width=0.8pt,color=ffqqqq] (-3.5996117462953032,4.466796084568232) -- (-3.457309251587726,4.37192775476318);
\draw [line width=0.8pt,color=ffqqqq] (-2.,5.)-- (-5.,3.);
\draw [line width=0.8pt,color=ffqqqq] (-3.587365280905473,3.9417564793963504) -- (-3.574884526490405,4.112326789735608);
\draw [line width=0.8pt,color=ffqqqq] (-3.587365280905473,3.9417564793963504) -- (-3.4251154735095932,3.8876732102643903);
\draw [line width=0.8pt,color=ffqqqq] (-2.,5.)-- (-5.,0.);
\draw [line width=0.8pt,color=ffqqqq] (-3.5540220543198897,2.409963242800183) -- (-3.615761544971192,2.5694569269827165);
\draw [line width=0.8pt,color=ffqqqq] (-3.5540220543198897,2.409963242800183) -- (-3.3842384550288056,2.4305430730172843);
\draw [line width=0.8pt,color=ffqqqq] (-2.,4.)-- (-5.,3.);
\draw [line width=0.8pt,color=ffqqqq] (-3.5996117462953032,3.466796084568232) -- (-3.542690748412273,3.6280722452368193);
\draw [line width=0.8pt,color=ffqqqq] (-3.5996117462953032,3.466796084568232) -- (-3.457309251587726,3.3719277547631803);
\draw [line width=0.8pt,color=ffqqqq] (-2.,4.)-- (-5.,0.);
\draw [line width=0.8pt,color=ffqqqq] (-3.563,1.916) -- (-3.608,2.081);
\draw [line width=0.8pt,color=ffqqqq] (-3.563,1.916) -- (-3.392,1.919);
\draw [line width=0.8pt,color=ffqqqq] (-2.,1.)-- (-5.,0.);
\draw [line width=0.8pt,color=ffqqqq] (-3.5996117462953032,0.46679608456823174) -- (-3.542690748412273,0.6280722452368194);
\draw [line width=0.8pt,color=ffqqqq] (-3.5996117462953032,0.46679608456823174) -- (-3.457309251587726,0.37192775476318035);
\draw [line width=0.4pt] (1.44,5.62)-- (2.6,5.6);
\draw [line width=0.4pt] (2.6,5.6)-- (2.6,-0.6);
\draw [line width=0.4pt] (2.6,-0.6)-- (1.4,-0.6);
\draw [line width=0.4pt] (1.4,-0.6)-- (1.44,5.62);
\draw [line width=0.4pt] (4.4,5.6)-- (5.6,5.6);
\draw [line width=0.4pt] (5.6,5.6)-- (5.6,-0.6);
\draw [line width=0.4pt] (5.6,-0.6)-- (4.4,-0.6);
\draw [line width=0.4pt] (4.4,-0.6)-- (4.4,5.6);
\draw [line width=0.8pt,color=qqqqff] (5.,5.)-- (5.,4.);
\draw [line width=0.8pt,color=qqqqff] (5.,4.395) -- (4.865,4.5);
\draw [line width=0.8pt,color=qqqqff] (5.,4.395) -- (5.135,4.5);
\draw [line width=0.8pt,color=qqqqff] (5.,1.)-- (5.,0.);
\draw [line width=0.8pt,color=qqqqff] (5.,0.395) -- (4.865,0.5);
\draw [line width=0.8pt,color=qqqqff] (5.,0.395) -- (5.135,0.5);
\draw [line width=0.8pt,color=ffqqqq] (5.,1.)-- (2.,0.);
\draw [line width=0.8pt,color=ffqqqq] (3.4003882537046968,0.46679608456823174) -- (3.4573092515877275,0.6280722452368194);
\draw [line width=0.8pt,color=ffqqqq] (3.4003882537046968,0.46679608456823174) -- (3.542690748412273,0.37192775476318035);
\draw [line width=0.8pt,color=ffqqqq] (5.,4.)-- (2.,0.);
\draw [line width=0.8pt,color=ffqqqq] (3.437,1.916) -- (3.392,2.081);
\draw [line width=0.8pt,color=ffqqqq] (3.437,1.916) -- (3.608,1.919);
\draw [line width=0.8pt,color=ffqqqq] (5.,4.)-- (2.,3.);
\draw [line width=0.8pt,color=ffqqqq] (3.4003882537046968,3.466796084568232) -- (3.4573092515877275,3.6280722452368193);
\draw [line width=0.8pt,color=ffqqqq] (3.4003882537046968,3.466796084568232) -- (3.542690748412273,3.3719277547631803);
\draw [line width=0.8pt,color=ffqqqq] (5.,5.)-- (2.,3.);
\draw [line width=0.8pt,color=ffqqqq] (3.4126347190945276,3.9417564793963504) -- (3.425115473509594,4.112326789735608);
\draw [line width=0.8pt,color=ffqqqq] (3.4126347190945276,3.9417564793963504) -- (3.5748845264904063,3.8876732102643903);
\draw [line width=0.8pt,color=ffqqqq] (5.,5.)-- (2.,4.);
\draw [line width=0.8pt,color=ffqqqq] (3.4003882537046968,4.466796084568232) -- (3.4573092515877275,4.628072245236819);
\draw [line width=0.8pt,color=ffqqqq] (3.4003882537046968,4.466796084568232) -- (3.542690748412273,4.37192775476318);
\draw [line width=0.8pt,color=ffqqqq] (5.,5.)-- (2.,0.);
\draw [line width=0.8pt,color=ffqqqq] (3.445977945680112,2.409963242800183) -- (3.384238455028809,2.5694569269827165);
\draw [line width=0.8pt,color=ffqqqq] (3.445977945680112,2.409963242800183) -- (3.6157615449711957,2.4305430730172843);
\draw [line width=0.8pt,color=qqqqff] (-5.,5.)-- (-5.,4.);
\draw [line width=0.8pt,color=qqqqff] (-5.,4.395) -- (-5.135,4.5);
\draw [line width=0.8pt,color=qqqqff] (-5.,4.395) -- (-4.865,4.5);
\draw [line width=0.8pt,color=qqqqff] (-5.,4.)-- (-5.,3.);
\draw [line width=0.8pt,color=qqqqff] (-5.,3.395) -- (-5.135,3.5);
\draw [line width=0.8pt,color=qqqqff] (-5.,3.395) -- (-4.865,3.5);
\draw [line width=0.8pt,color=qqqqff] (-5.,0.8)-- (-5.,0.);
\draw [line width=0.8pt,color=qqqqff] (-5.,0.295) -- (-5.135,0.4);
\draw [line width=0.8pt,color=qqqqff] (-5.,0.295) -- (-4.865,0.4);
%\draw [line width=0.8pt,color=qqqqff] (2.,5.)-- (2.,4.);
%\draw [line width=0.8pt,color=qqqqff] (2.,4.395) -- (1.865,4.5);
%\draw [line width=0.8pt,color=qqqqff] (2.,4.395) -- (2.135,4.5);
%\draw [line width=0.8pt,color=qqqqff] (2.,4.)-- (2.,3.);
%\draw [line width=0.8pt,color=qqqqff] (2.,3.395) -- (1.865,3.5);
%\draw [line width=0.8pt,color=qqqqff] (2.,3.395) -- (2.135,3.5);
\draw [line width=0.8pt,color=qqqqff] (-5.,3.)-- (-5.,2.4);
\draw [line width=0.8pt,color=qqqqff] (-5.,2.595) -- (-5.135,2.7);
\draw [line width=0.8pt,color=qqqqff] (-5.,2.595) -- (-4.865,2.7);
\draw [line width=0.8pt,color=qqqqff] (5.,4.)-- (5.,3.4);
\draw [line width=0.8pt,color=qqqqff] (5.,3.595) -- (4.865,3.7);
\draw [line width=0.8pt,color=qqqqff] (5.,3.595) -- (5.135,3.7);
%\draw [line width=0.8pt,color=qqqqff] (2.,0.8)-- (2.,0.);
%\draw [line width=0.8pt,color=qqqqff] (2.,0.295) -- (1.865,0.4);
%\draw [line width=0.8pt,color=qqqqff] (2.,0.295) -- (2.135,0.4);
%\draw [line width=0.8pt,color=qqqqff] (2.,3.)-- (2.,2.4);
%\draw [line width=0.8pt,color=qqqqff] (2.,2.595) -- (1.865,2.7);
%\draw [line width=0.8pt,color=qqqqff] (2.,2.595) -- (2.135,2.7);
\draw [shift={(0.,-4.)},line width=0.8pt,color=ffqqqq]  plot[domain=1.0906212971913698:2.0526087419383687,variable=\t]({1.*10.824047302187846*cos(\t r)+0.*10.824047302187846*sin(\t r)},{0.*10.824047302187846*cos(\t r)+1.*10.824047302187846*sin(\t r)});
\draw [shift={(1.58,-2.02)},line width=0.8pt,color=ffqqqq]  plot[domain=1.148925040371887:1.9948500778320195,variable=\t]({1.*8.352293098305397*cos(\t r)+0.*8.352293098305397*sin(\t r)},{0.*8.352293098305397*cos(\t r)+1.*8.352293098305397*sin(\t r)});
\draw [shift={(-1.56,8.76)},line width=0.8pt,color=ffqqqq]  plot[domain=4.3601909854188765:5.062703192620091,variable=\t]({1.*9.972121138453945*cos(\t r)+0.*9.972121138453945*sin(\t r)},{0.*9.972121138453945*cos(\t r)+1.*9.972121138453945*sin(\t r)});
%\draw [line width=1.6pt,color=ffqqqq] (-1.31929459167295,-1.2092156615454002)-- (-1.72081974660243,-1.2108242893505423);
\draw [line width=1.6pt,color=ffqqqq] (-1.6550560857493173,-1.210560820792924) -- (-1.520718202337079,-1.0450212995893147);
\draw [line width=1.6pt,color=ffqqqq] (-1.6550560857493173,-1.210560820792924) -- (-1.519396135938303,-1.3750186513066236);
%\draw [line width=1.6pt,color=ffqqqq] (0.677667231546819,6.802812926422806)-- (0.19967168762325777,6.822205469180563);
\draw [line width=1.6pt,color=ffqqqq] (0.30378042556929596,6.817981720282189) -- (0.44535809817232364,6.977373572709816);
\draw [line width=1.6pt,color=ffqqqq] (0.30378042556929596,6.817981720282189) -- (0.43198082099775564,6.647644822893553);
%\draw [line width=1.6pt,color=ffqqqq] (0.7843721022175187,6.294311531826926)-- (0.20689147296106603,6.218651162233595);
\draw [line width=1.6pt,color=ffqqqq] (0.36177576679162954,6.238943795997798) -- (0.4741970029940582,6.420083150227401);
\draw [line width=1.6pt,color=ffqqqq] (0.36177576679162954,6.238943795997798) -- (0.5170665721845218,6.092879543833119);
\draw [line width=0.8pt,color=qqqqff] (5.,1.8)-- (5.,1.);
\draw [line width=0.8pt,color=qqqqff] (5.,1.295) -- (4.865,1.4);
\draw [line width=0.8pt,color=qqqqff] (5.,1.295) -- (5.135,1.4);

\draw [line width=0.8pt,color=ffqqqq] (5.02,-0.6)-- (0.,-2.6);
\draw [line width=0.8pt,color=ffqqqq] (2.4124564368305195,-1.6388619773583588) -- (2.460034600539253,-1.4745868473535262);
\draw [line width=0.8pt,color=ffqqqq] (2.4124564368305195,-1.6388619773583588) -- (2.5599653994607467,-1.7254131526464749);

\draw [line width=0.8pt,color=ffqqqq] (0.,-2.6)-- (-5.,-0.6);
\draw [line width=0.8pt,color=ffqqqq] (-2.597490052542952,-1.5610039789828198) -- (-2.449862258692195,-1.4746556467304919);
\draw [line width=0.8pt,color=ffqqqq] (-2.597490052542952,-1.5610039789828198) -- (-2.5501377413078035,-1.7253443532695116);

\begin{scriptsize}
\draw [fill=black] (-5.,5.) circle (1.5pt);
\draw[color=black] (-5.2,5.1) node {$1$};
\draw [fill=black] (-5.,4.) circle (1.5pt);
\draw[color=black] (-5.2,4.1) node {$3$};
\draw [fill=black] (-5.,0.) circle (1.5pt);
\draw[color=black] (-5.1,-0.2) node {$2n_1-1$};
\draw [fill=black] (-2.,5.) circle (1.5pt);
\draw[color=black] (-1.8,5.1) node {$2$};
\draw [fill=black] (-2.,4.) circle (1.5pt);
\draw[color=black] (-1.8,4.1) node {$4$};
\draw [fill=black] (-2.,1.) circle (1.5pt);
\draw[color=black] (-1.9,1.29) node {$2n_1-2$};
\draw [fill=black] (-2.,0.) circle (1.5pt);
\draw[color=black] (-1.9,-0.3) node {$2n_1$};
\draw [fill=black] (-5.,3.) circle (1.5pt);
\draw[color=black] (-5.2,3.1) node {$5$};
\draw [fill=black] (2.,5.) circle (1.5pt);
\draw[color=black] (1.7,5.2) node {$2s_2$};
\draw [fill=black] (2.,4.) circle (1.5pt);
\draw[color=black] (1.8,4.2) node {$2s_2-2$};
\draw [fill=black] (2.,3.) circle (1.5pt);
\draw[color=black] (1.8,3.2) node {$2s_2-4$};
\draw [fill=black] (2.,0.) circle (1.5pt);
\draw[color=black] (1.9,-0.2) node {$2n_1+2$};
\draw [fill=black] (5.,5.) circle (1.5pt);
\draw[color=black] (5,5.2) node {$2s_2-1$};
\draw [fill=black] (5.,4.) circle (1.5pt);
\draw[color=black] (5,4.2) node {$2s_2-3$};
\draw [fill=black] (5.,1.) circle (1.5pt);
\draw[color=black] (5,1.2) node {$2n_1+3$};
\draw [fill=black] (5.,0.) circle (1.5pt);
\draw[color=black] (4.9,-0.2) node {$2n_1+1$};
\draw [fill=black] (-5.,2.) circle (0.5pt);
\draw [fill=black] (-5.,1.6) circle (0.5pt);
\draw [fill=black] (-5.,1.2) circle (0.5pt);
\draw [fill=black] (-2.,3.2) circle (0.5pt);
\draw [fill=black] (-2.,2.6) circle (0.5pt);
\draw [fill=black] (-2.,1.8) circle (0.5pt);
\draw [fill=black] (2.,2.) circle (0.5pt);
\draw [fill=black] (2.,1.6) circle (0.5pt);
\draw [fill=black] (2.,1.2) circle (0.5pt);
\draw [fill=black] (5.,2.2) circle (0.5pt);
\draw [fill=black] (5.,3.) circle (0.5pt);
\draw [fill=black] (5.,2.6) circle (0.5pt);

\draw [fill=black] (0.,-2.6) circle (2.5pt);
\draw[color=black] (0.1,-2.26) node {$0$};

\draw[color=black] (-3.1,-0.95) node {Le poset $t_1(R'_{2n_1})$};
\draw[color=black] (3.5,-0.95) node {Le poset  $t_{2n_1+1}((R'_{2n_2})^{\star})$};
\end{scriptsize}

\end{tikzpicture}

\caption{\textbf{Le poset (-1)-critique $R'_{1,2n_1,2n_2}$} (0 est le sommet non critique).}  
\label{poset (-1)-crit R'}
\end{center} 
\end{figure}
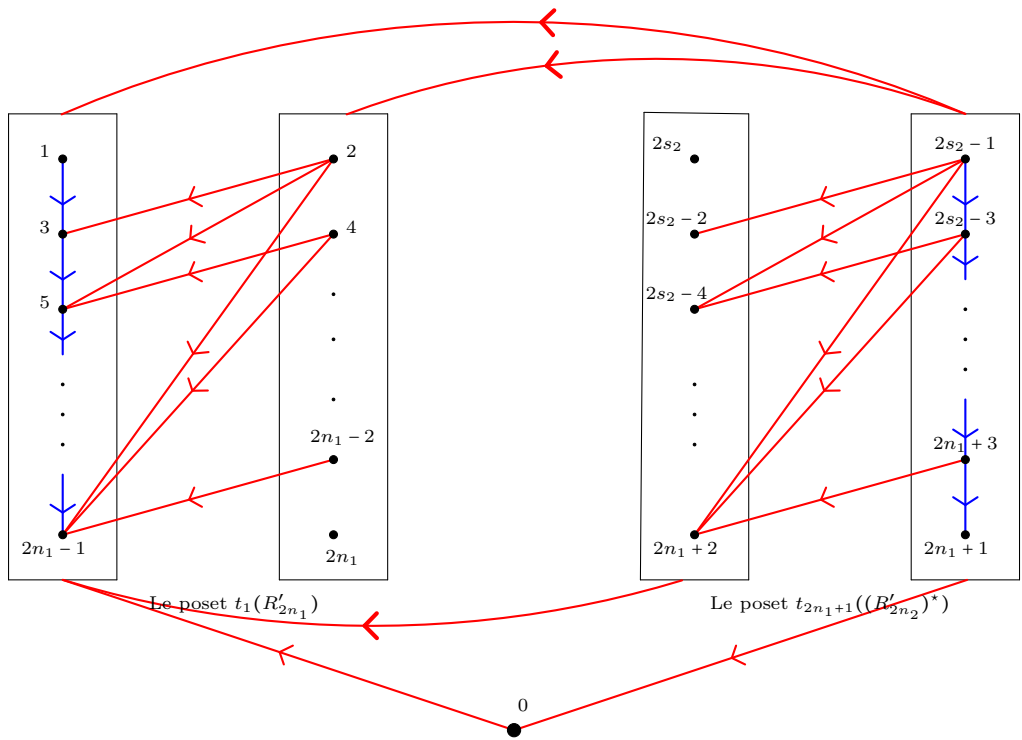 
 
  \begin{Fait}\label{R' est -1critique} Pour tous entiers strictement positifs $n_1$ et $n_2$, le digraphe $R'_{1, 2n_1,  2n_2}$ est une orientation transitive du graphe  $\overline{H_{1,2n_1,2n_2}}$. Le graphe (-1)-critique $\overline{H_{1,2n_1,2n_2}}$ est alors un graphe de comparabilité.
 \end{Fait}
 
% Notons que $R_{1, 2,  2}\simeq R'_{1, 2,  2}$. 
 
 Enfin, il reste a voir, si les graphes complémentaires de $G_{2n_1, \ldots, 2n_k}$ et de $G'_{2n_1, \ldots, 2n_k}$ sont de comparabilité. Pour tout $k\geqslant 3$, Nous conjecturons que ces derniers graphes ne sont pas de comparabilité. Pour $k=2$, nous construisons, ci-après, une orientation transitive du graphe complémentaire de $G'_{2n_1, 2n_2}$, montrant ainsi la comparabilité de ce graphe (-1)-critique. 
 
 Pour tout entier $n\geqslant 1$, le digraphe transitif $R_{2n}$ est défini sur  $\{0\ldots,2n-1\}$ par $A(R_{2n})=A(R'_{2n})\cup \{ (2p+1,2q+1): 0\leqslant p< q 
 \leqslant n-1\}$. Pour tous entiers $n_1$ et $n_2$ tels que $n_1\geqslant 2$ et $n_2\geqslant 1$, désignons par $O_{2n_1,2n_2}$ le digraphe défini sur 
 $\{0, \ldots, 2s_2-1\}$ par $A(O_{ 2n_1, 2n_2}) =  A(R'_{2n_1})\cup A(t_{2n_1}((R_{2n_2})^{\star}))\cup \{(2q,p): 0 \leqslant p \leqslant 2 n_1-1 \ \text{et} \  n_1\leqslant q \leqslant~ s_2-1\} \cup  \{(2q+1,2p): 0 \leqslant p \leqslant n_1-1 \ \text{et} \  n_1\leqslant q \leqslant s_2 -1\}$ (voir Figure~\ref{poset (-1)-crit O}). Nous vérifions facilement le résultat suivant. 
 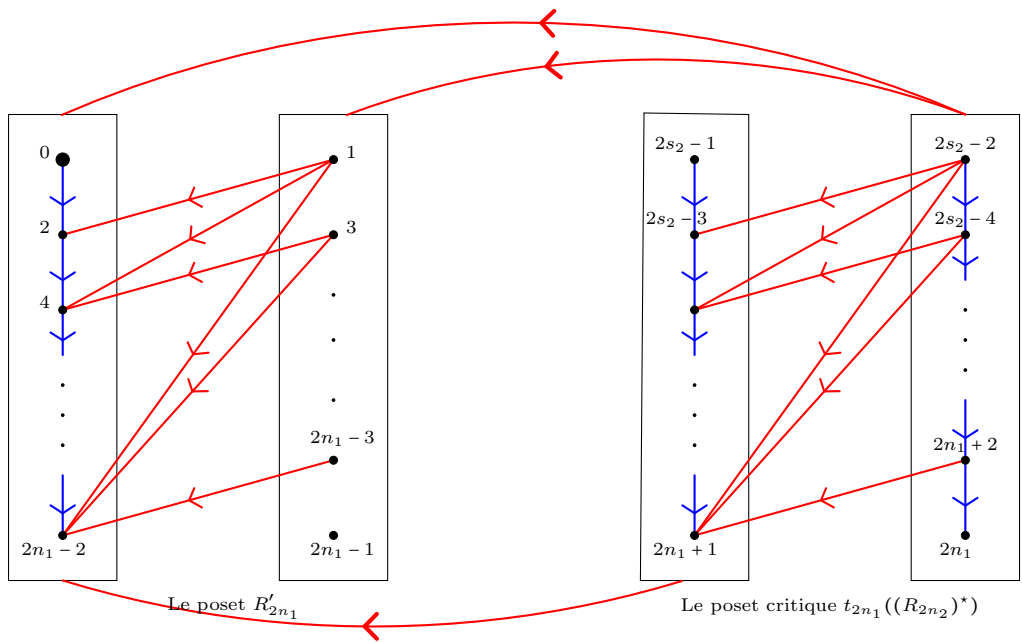
\begin{figure}[h]
\begin{center}
\definecolor{qqqqff}{rgb}{0.,0.,1.}
\definecolor{ffqqqq}{rgb}{1.,0.,0.}
\begin{tikzpicture}[line cap=round,line join=round,>=triangle 45,x=1.2cm,y=1.0cm]
%\clip(-11.38,-6.98) rectangle (21.7,9.58);
\draw [line width=0.4pt] (-5.6,5.6)-- (-4.4,5.6);
\draw [line width=0.4pt] (-4.4,5.6)-- (-4.4,-0.6);
\draw [line width=0.4pt] (-5.6,5.6)-- (-5.6,-0.6);
\draw [line width=0.4pt] (-5.6,-0.6)-- (-4.4,-0.6);
\draw [line width=0.4pt] (-2.6,5.6)-- (-1.4,5.6);
\draw [line width=0.4pt] (-1.4,5.6)-- (-1.4,-0.6);
\draw [line width=0.4pt] (-1.4,-0.6)-- (-2.6,-0.6);
\draw [line width=0.4pt] (-2.6,-0.6)-- (-2.6,5.6);
\draw [line width=0.8pt,color=ffqqqq] (-2.,5.)-- (-5.,4.);
\draw [line width=0.8pt,color=ffqqqq] (-3.5996117462953032,4.466796084568232) -- (-3.542690748412273,4.628072245236819);
\draw [line width=0.8pt,color=ffqqqq] (-3.5996117462953032,4.466796084568232) -- (-3.457309251587726,4.37192775476318);
\draw [line width=0.8pt,color=ffqqqq] (-2.,5.)-- (-5.,3.);
\draw [line width=0.8pt,color=ffqqqq] (-3.587365280905473,3.9417564793963504) -- (-3.574884526490405,4.112326789735608);
\draw [line width=0.8pt,color=ffqqqq] (-3.587365280905473,3.9417564793963504) -- (-3.4251154735095932,3.8876732102643903);
\draw [line width=0.8pt,color=ffqqqq] (-2.,5.)-- (-5.,0.);
\draw [line width=0.8pt,color=ffqqqq] (-3.5540220543198897,2.409963242800183) -- (-3.615761544971192,2.5694569269827165);
\draw [line width=0.8pt,color=ffqqqq] (-3.5540220543198897,2.409963242800183) -- (-3.3842384550288056,2.4305430730172843);
\draw [line width=0.8pt,color=ffqqqq] (-2.,4.)-- (-5.,3.);
\draw [line width=0.8pt,color=ffqqqq] (-3.5996117462953032,3.466796084568232) -- (-3.542690748412273,3.6280722452368193);
\draw [line width=0.8pt,color=ffqqqq] (-3.5996117462953032,3.466796084568232) -- (-3.457309251587726,3.3719277547631803);
\draw [line width=0.8pt,color=ffqqqq] (-2.,4.)-- (-5.,0.);
\draw [line width=0.8pt,color=ffqqqq] (-3.563,1.916) -- (-3.608,2.081);
\draw [line width=0.8pt,color=ffqqqq] (-3.563,1.916) -- (-3.392,1.919);
\draw [line width=0.8pt,color=ffqqqq] (-2.,1.)-- (-5.,0.);
\draw [line width=0.8pt,color=ffqqqq] (-3.5996117462953032,0.46679608456823174) -- (-3.542690748412273,0.6280722452368194);
\draw [line width=0.8pt,color=ffqqqq] (-3.5996117462953032,0.46679608456823174) -- (-3.457309251587726,0.37192775476318035);
\draw [line width=0.4pt] (1.44,5.62)-- (2.6,5.6);
\draw [line width=0.4pt] (2.6,5.6)-- (2.6,-0.6);
\draw [line width=0.4pt] (2.6,-0.6)-- (1.4,-0.6);
\draw [line width=0.4pt] (1.4,-0.6)-- (1.44,5.62);
\draw [line width=0.4pt] (4.4,5.6)-- (5.6,5.6);
\draw [line width=0.4pt] (5.6,5.6)-- (5.6,-0.6);
\draw [line width=0.4pt] (5.6,-0.6)-- (4.4,-0.6);
\draw [line width=0.4pt] (4.4,-0.6)-- (4.4,5.6);
\draw [line width=0.8pt,color=qqqqff] (5.,5.)-- (5.,4.);
\draw [line width=0.8pt,color=qqqqff] (5.,4.395) -- (4.865,4.5);
\draw [line width=0.8pt,color=qqqqff] (5.,4.395) -- (5.135,4.5);
\draw [line width=0.8pt,color=qqqqff] (5.,1.)-- (5.,0.);
\draw [line width=0.8pt,color=qqqqff] (5.,0.395) -- (4.865,0.5);
\draw [line width=0.8pt,color=qqqqff] (5.,0.395) -- (5.135,0.5);
\draw [line width=0.8pt,color=ffqqqq] (5.,1.)-- (2.,0.);
\draw [line width=0.8pt,color=ffqqqq] (3.4003882537046968,0.46679608456823174) -- (3.4573092515877275,0.6280722452368194);
\draw [line width=0.8pt,color=ffqqqq] (3.4003882537046968,0.46679608456823174) -- (3.542690748412273,0.37192775476318035);
\draw [line width=0.8pt,color=ffqqqq] (5.,4.)-- (2.,0.);
\draw [line width=0.8pt,color=ffqqqq] (3.437,1.916) -- (3.392,2.081);
\draw [line width=0.8pt,color=ffqqqq] (3.437,1.916) -- (3.608,1.919);
\draw [line width=0.8pt,color=ffqqqq] (5.,4.)-- (2.,3.);
\draw [line width=0.8pt,color=ffqqqq] (3.4003882537046968,3.466796084568232) -- (3.4573092515877275,3.6280722452368193);
\draw [line width=0.8pt,color=ffqqqq] (3.4003882537046968,3.466796084568232) -- (3.542690748412273,3.3719277547631803);
\draw [line width=0.8pt,color=ffqqqq] (5.,5.)-- (2.,3.);
\draw [line width=0.8pt,color=ffqqqq] (3.4126347190945276,3.9417564793963504) -- (3.425115473509594,4.112326789735608);
\draw [line width=0.8pt,color=ffqqqq] (3.4126347190945276,3.9417564793963504) -- (3.5748845264904063,3.8876732102643903);
\draw [line width=0.8pt,color=ffqqqq] (5.,5.)-- (2.,4.);
\draw [line width=0.8pt,color=ffqqqq] (3.4003882537046968,4.466796084568232) -- (3.4573092515877275,4.628072245236819);
\draw [line width=0.8pt,color=ffqqqq] (3.4003882537046968,4.466796084568232) -- (3.542690748412273,4.37192775476318);
\draw [line width=0.8pt,color=ffqqqq] (5.,5.)-- (2.,0.);
\draw [line width=0.8pt,color=ffqqqq] (3.445977945680112,2.409963242800183) -- (3.384238455028809,2.5694569269827165);
\draw [line width=0.8pt,color=ffqqqq] (3.445977945680112,2.409963242800183) -- (3.6157615449711957,2.4305430730172843);
\draw [line width=0.8pt,color=qqqqff] (-5.,5.)-- (-5.,4.);
\draw [line width=0.8pt,color=qqqqff] (-5.,4.395) -- (-5.135,4.5);
\draw [line width=0.8pt,color=qqqqff] (-5.,4.395) -- (-4.865,4.5);
\draw [line width=0.8pt,color=qqqqff] (-5.,4.)-- (-5.,3.);
\draw [line width=0.8pt,color=qqqqff] (-5.,3.395) -- (-5.135,3.5);
\draw [line width=0.8pt,color=qqqqff] (-5.,3.395) -- (-4.865,3.5);
\draw [line width=0.8pt,color=qqqqff] (-5.,0.8)-- (-5.,0.);
\draw [line width=0.8pt,color=qqqqff] (-5.,0.295) -- (-5.135,0.4);
\draw [line width=0.8pt,color=qqqqff] (-5.,0.295) -- (-4.865,0.4);
\draw [line width=0.8pt,color=qqqqff] (2.,5.)-- (2.,4.);
\draw [line width=0.8pt,color=qqqqff] (2.,4.395) -- (1.865,4.5);
\draw [line width=0.8pt,color=qqqqff] (2.,4.395) -- (2.135,4.5);
\draw [line width=0.8pt,color=qqqqff] (2.,4.)-- (2.,3.);
\draw [line width=0.8pt,color=qqqqff] (2.,3.395) -- (1.865,3.5);
\draw [line width=0.8pt,color=qqqqff] (2.,3.395) -- (2.135,3.5);
\draw [line width=0.8pt,color=qqqqff] (-5.,3.)-- (-5.,2.4);
\draw [line width=0.8pt,color=qqqqff] (-5.,2.595) -- (-5.135,2.7);
\draw [line width=0.8pt,color=qqqqff] (-5.,2.595) -- (-4.865,2.7);
\draw [line width=0.8pt,color=qqqqff] (5.,4.)-- (5.,3.4);
\draw [line width=0.8pt,color=qqqqff] (5.,3.595) -- (4.865,3.7);
\draw [line width=0.8pt,color=qqqqff] (5.,3.595) -- (5.135,3.7);
\draw [line width=0.8pt,color=qqqqff] (2.,0.8)-- (2.,0.);
\draw [line width=0.8pt,color=qqqqff] (2.,0.295) -- (1.865,0.4);
\draw [line width=0.8pt,color=qqqqff] (2.,0.295) -- (2.135,0.4);
\draw [line width=0.8pt,color=qqqqff] (2.,3.)-- (2.,2.4);
\draw [line width=0.8pt,color=qqqqff] (2.,2.595) -- (1.865,2.7);
\draw [line width=0.8pt,color=qqqqff] (2.,2.595) -- (2.135,2.7);
\draw [shift={(0.,-4.)},line width=0.8pt,color=ffqqqq]  plot[domain=1.0906212971913698:2.0526087419383687,variable=\t]({1.*10.824047302187846*cos(\t r)+0.*10.824047302187846*sin(\t r)},{0.*10.824047302187846*cos(\t r)+1.*10.824047302187846*sin(\t r)});
\draw [shift={(1.58,-2.02)},line width=0.8pt,color=ffqqqq]  plot[domain=1.148925040371887:1.9948500778320195,variable=\t]({1.*8.352293098305397*cos(\t r)+0.*8.352293098305397*sin(\t r)},{0.*8.352293098305397*cos(\t r)+1.*8.352293098305397*sin(\t r)});
\draw [shift={(-1.56,8.76)},line width=0.8pt,color=ffqqqq]  plot[domain=4.3601909854188765:5.062703192620091,variable=\t]({1.*9.972121138453945*cos(\t r)+0.*9.972121138453945*sin(\t r)},{0.*9.972121138453945*cos(\t r)+1.*9.972121138453945*sin(\t r)});
%\draw [line width=1.6pt,color=ffqqqq] (-1.31929459167295,-1.2092156615454002)-- (-1.72081974660243,-1.2108242893505423);
\draw [line width=1.6pt,color=ffqqqq] (-1.6550560857493173,-1.210560820792924) -- (-1.520718202337079,-1.0450212995893147);
\draw [line width=1.6pt,color=ffqqqq] (-1.6550560857493173,-1.210560820792924) -- (-1.519396135938303,-1.3750186513066236);
%\draw [line width=1.6pt,color=ffqqqq] (0.677667231546819,6.802812926422806)-- (0.19967168762325777,6.822205469180563);
\draw [line width=1.6pt,color=ffqqqq] (0.30378042556929596,6.817981720282189) -- (0.44535809817232364,6.977373572709816);
\draw [line width=1.6pt,color=ffqqqq] (0.30378042556929596,6.817981720282189) -- (0.43198082099775564,6.647644822893553);
%\draw [line width=1.6pt,color=ffqqqq] (0.7843721022175187,6.294311531826926)-- (0.20689147296106603,6.218651162233595);
\draw [line width=1.6pt,color=ffqqqq] (0.36177576679162954,6.238943795997798) -- (0.4741970029940582,6.420083150227401);
\draw [line width=1.6pt,color=ffqqqq] (0.36177576679162954,6.238943795997798) -- (0.5170665721845218,6.092879543833119);
\draw [line width=0.8pt,color=qqqqff] (5.,1.8)-- (5.,1.);
\draw [line width=0.8pt,color=qqqqff] (5.,1.295) -- (4.865,1.4);
\draw [line width=0.8pt,color=qqqqff] (5.,1.295) -- (5.135,1.4);
\begin{scriptsize}
\draw [fill=black] (-5.,5.) circle (2.5pt);
\draw[color=black] (-5.2,5.1) node {$0$};
\draw [fill=black] (-5.,4.) circle (1.5pt);
\draw[color=black] (-5.2,4.1) node {$2$};
\draw [fill=black] (-5.,0.) circle (1.5pt);
\draw[color=black] (-5.1,-0.2) node {$2n_1-2$};
\draw [fill=black] (-2.,5.) circle (1.5pt);
\draw[color=black] (-1.8,5.1) node {$1$};
\draw [fill=black] (-2.,4.) circle (1.5pt);
\draw[color=black] (-1.8,4.1) node {$3$};
\draw [fill=black] (-2.,1.) circle (1.5pt);
\draw[color=black] (-1.9,1.29) node {$2n_1-3$};
\draw [fill=black] (-2.,0.) circle (1.5pt);
\draw[color=black] (-1.9,-0.2) node {$2n_1-1$};
\draw [fill=black] (-5.,3.) circle (1.5pt);
\draw[color=black] (-5.2,3.1) node {$4$};
\draw [fill=black] (2.,5.) circle (1.5pt);
\draw[color=black] (1.9,5.2) node {$2s_2-1$};
\draw [fill=black] (2.,4.) circle (1.5pt);
\draw[color=black] (1.8,4.2) node {$2s_2-3$};
\draw [fill=black] (2.,3.) circle (1.5pt);
%\draw[color=black] (1.9,3.2) node {$2s_2-5$};
\draw [fill=black] (2.,0.) circle (1.5pt);
\draw[color=black] (1.9,-0.2) node {$2n_1+1$};
\draw [fill=black] (5.,5.) circle (1.5pt);
\draw[color=black] (5,5.2) node {$2s_2-2$};
\draw [fill=black] (5.,4.) circle (1.5pt);
\draw[color=black] (5,4.2) node {$2s_2-4$};
\draw [fill=black] (5.,1.) circle (1.5pt);
\draw[color=black] (5,1.2) node {$2n_1+2$};
\draw [fill=black] (5.,0.) circle (1.5pt);
\draw[color=black] (4.9,-0.2) node {$2n_1$};
\draw [fill=black] (-5.,2.) circle (0.5pt);
\draw [fill=black] (-5.,1.6) circle (0.5pt);
\draw [fill=black] (-5.,1.2) circle (0.5pt);
\draw [fill=black] (-2.,3.2) circle (0.5pt);
\draw [fill=black] (-2.,2.6) circle (0.5pt);
\draw [fill=black] (-2.,1.8) circle (0.5pt);
\draw [fill=black] (2.,2.) circle (0.5pt);
\draw [fill=black] (2.,1.6) circle (0.5pt);
\draw [fill=black] (2.,1.2) circle (0.5pt);
\draw [fill=black] (5.,2.2) circle (0.5pt);
\draw [fill=black] (5.,3.) circle (0.5pt);
\draw [fill=black] (5.,2.6) circle (0.5pt);
\draw[color=black] (-3.1,-0.95) node {Le poset $R'_{2n_1}$};
\draw[color=black] (3.5,-0.95) node {Le poset critique $t_{2n_1}((R_{2n_2})^{\star})$};
\end{scriptsize}
\end{tikzpicture}
\caption{\textbf{Le poset (-1)-critique $O_{2n_1,2n_2}$} (0 est le sommet non critique).}  
\label{poset (-1)-crit O}
\end{center} 
\end{figure}
 
\begin{Fait}\label{O est -1critique} Pour tous entiers  $n_1$ et $n_2$ tels que $n_1\geqslant 2$ et $n_2\geqslant 1$,  le digraphe $O_{ 2n_1,  2n_2}$ est une orientation transitive du graphe $\overline{G'_{2n_1,2n_2}}$. Le graphe (-1)-critique $\overline{G'_{2n_1,2n_2}}$ est alors un graphe de comparabilité.
 \end{Fait}

Introduisons maintenant, les familles suivantes: 
 $$\mathcal{Q} = \{Q_{2n_1, \ldots, 2n_k}: \ k \geqslant 3 \ \textnormal{et pour tout} \ i \in \{1, \ldots, k\},  \ n_i \geqslant 1\}, $$
$$\mathcal{Q'} = \{Q'_{2n_1, \ldots, 2n_k}: \ k \geqslant 2, \ n_1 \geqslant 2, \ \textnormal{et pour tout} \ i \in \{2, \ldots, k\}, \  n_i \geqslant 1\}, $$ 
$$\mathcal{O} = \{O_{ 2n_1, 2n_2}:    n_1 \geqslant 2, n_2 \geqslant 1\}$$
$$\mathcal{R} = \{R_{1, 2n_1, 2n_2}:    n_1 \geqslant 1, n_2 \geqslant 1\} \; {\rm \ et} \; \mathcal{R'} = \{R'_{1, 2n_1, 2n_2}:    n_1 \geqslant 1, n_2 \geqslant 1\}.$$
Rappelons enfin, le résultat suivant.
\begin{lemme}\cite{livre de Golumbic}\label{exactement deux orientations transi} Soit $G$ un graphe de comparabilité. Si $G$ est indécomposable, alors $G$ admet exactement deux orientations transitives, l'une étant le dual de l'autre.
\end{lemme}

D'après Corollaire~\ref{corollaire des posets (-1)-critiques}, Théorème~\ref{graphes (-1)-critiques}, les lemmes~\ref{k ne depasse pas 3 pour H} et \ref{exactement deux orientations transi}, et des  faits~\ref{le post -1critique Q},  \ref{le post -1critique Q'},  \ref{R est un poset -1critique}, \ref{R' est -1critique} et \ref{O est -1critique}, nous obtenons la caractérisation morphologique suivante des posets (-1)-critiques.

\begin{theoreme} \label{poset (-1)-critiques}
 \`A isomorphisme pr\`es, les posets (-1)-critiques sont les digraphes de $\mathcal{Q} \cup \mathcal{Q'}\cup \mathcal{O}\cup \mathcal{R}\cup \mathcal{R'}$ et leurs duaux, le sommet $0$ \'etant l'unique sommet non critique de chacun de ces posets.
\end{theoreme}

 \end{document}